\documentclass[11pt,a4paper,twoside]{amsart}

\usepackage{graphicx}
\usepackage{amsfonts}
\usepackage{amssymb}
\usepackage{amsthm}
\usepackage[centertags]{amsmath}
\usepackage{newlfont}
\usepackage{hyperref}
\usepackage{color}
\usepackage{graphicx,color}
\usepackage{amsmath, amssymb, graphics}

 \def\r{\mathbb{R}}
 \def\l{\mathbb{L}}
 
\newtheorem{theorem}{Theorem}[section]
\newtheorem{proposition}[theorem]{Proposition}
\newtheorem{corollary}[theorem]{Corollary}
\newtheorem{lemma}[theorem]{Lemma}

\theoremstyle{definition}
\newtheorem{definition}[theorem]{Definition}

\newtheorem{remark}[theorem]{Remark}

\renewcommand{\r}{\mathbb R}
\renewcommand{\l}{\mathbb L}
\renewcommand{\c}{\mathbb C}

\newcommand{\E}{\mathbf E}

\title{On intrinsic rotational surfaces in the Lorentz-Minkowski space} 
\author{Seher Kaya}  
\address{Accounting and Tax Department, Elmadag Vocational School\\
Ankara University, 06780. Ankara, T\"{u}rkiye} 
\email{Seher.Kaya@ankara.edu.tr}  
\author{Rafael L\'opez}  
\address{Departamento de Geometr\'{\i}a y Topolog\'{\i}a\\  Universidad de Granada. 18071 Granada, Spain} 
\email{rcamino@ugr.es}  



\begin{document} 
 
\begin{abstract}  
   Spacelike intrinsic rotational surfaces with constant mean curvature in the Lorentz-Minkowski space $\E_1^3$  have been recently investigated by Brander et al., extending the known Smyth's surfaces in Euclidean space. Assuming that the surface is intrinsic rotational with coordinates $(u,v)$ and conformal factor $\rho(u)^2$, we replace  the constancy of the mean curvature with the property that the Weingarten endomorphism $A$ can be expressed as $\Phi_{-\alpha(v)}\left(\begin{array}{ll}\lambda_1(u)&0\\ 0&\lambda_2(u)\end{array}\right)\Phi_{\alpha(v)}$, where $\Phi_{\alpha(v)}$ is the (Euclidean or hyperbolic) rotation of angle $\alpha(v)$ at each tangent plane and $\lambda_i$ are the principal curvatures. Under these conditions, it is  proved that the mean curvature is constant and $\alpha$ is a linear function. This result also covers the case that the surface is timelike. If the mean curvature is zero, we determine all spacelike and timelike intrinsic rotational surfaces with rotational angle $\alpha$. This family of surfaces includes   the spacelike and  timelike Enneper surfaces.
\end{abstract} 
\maketitle


\section{Introduction and motivation}\label{sec1}

An intrinsic rotational surface in Euclidean space is a surface that can be parametrized in local coordinates $(u,v)$ such that the metric is $ \rho(u)^2(du^2+dv^2)$ for some function $\rho(u)>0$. Surfaces of revolution are trivial examples of intrinsic rotational surfaces. In 1993, Smyth classified all intrinsic rotational surfaces with nonzero constant mean curvature \cite{sm}, extending the Delaunay surfaces: see also \cite[Appendix]{bo} and \cite{ti}. Later, in 2010, Brander, Rossman and Schmitt generalized the Smyth's surfaces to spacelike intrinsic rotational surfaces with nonzero constant mean curvature in the Lorentz-Minkowski space \cite{brs}. They constructed these surfaces using the  DPW method and studied some of  their properties: see also    \cite{og,og2} when the ambient space is   de Sitter and anti-de Sitter spaces. 

More recently, Freese and Weber gave a new approach to the Smyth's surfaces in the class of intrinsic rotational surfaces \cite{fw}. They were motivated by the Enneper surface, a well known minimal surface, which  is an intrinsic rotational surface. With the above coordinates $(u,v)$, 
they replace the constancy of the mean curvature with the property that the principal curvatures only depend on $u$ and the principal directions only depend on the angle of rotation $v$. Then it is proved that the surface has constant mean curvature and   the rotational speed of the principal curvature directions is constant \cite{fw}.

The purpose of this paper is to extend the approach of Freese and Weber to the Lorentz-Minkowski space $\E_1^3$. We will also investigate timelike intrinsic rotational surfaces, which were not considered in \cite{brs}. The Lorentz-Minkowski space  $\E_1^3$ is the vector space $\r^3$ endowed with the Lorentzian metric $\langle,\rangle  =dx_1^2+dx_2^2-dx_3^2$ where $(x_1,x_2,x_3)$ are the canonical coordinates of $\r^3$. The induced metric of a non-degenerate surface $\Sigma$ in $\E_1^3$ is   Riemannian   (the surface is called spacelike) or  Lorentzian   (the surface is called timelike).

\begin{definition}\label{def1}
 A non-degenerate surface $\Sigma$ of $\E_1^3$ is said to be an intrinsic rotational surface if there is a local coordinate system $(u,v)$ on $\Sigma$ such that the first fundamental form can be expressed as
\begin{equation}\label{in1}
\mathrm{I}=\rho(u)^2\left(\begin{array}{ll}\delta&0\\ 0&\epsilon\end{array}\right), \end{equation}
where $\delta,\epsilon\in\{-1,1\}$.
\end{definition}

Notice that   $\delta$ and $\epsilon$ cannot be simultaneously $-1$ in \eqref{in1}. The first examples of intrinsic rotational surfaces are the surfaces of revolution. The family of surfaces of revolution in $\E_1^3$ is richer than in Euclidean space because the rotational axis can be of three different causal characters.

\begin{proposition}\label{pr1}
Any non-degenerate surface of revolution in $\E_1^3$ is intrinsic rotational. 
\end{proposition}
\begin{proof} We distinguish the three cases according to the causal character of their rotation axis \cite{lop}.
\begin{enumerate}
\item Timelike rotation axis. Without loss of generality, we can assume that the rotation axis is the $x_3$-axis. The surface is parametrized by
  $$X(r,v)=(r\cos{v},r\sin{v},f(r)),\quad r\in I\subset\r, v\in\r,$$
  for some function $f$.  Let us change $r$ by a function $r(u)$ to determine. The   first fundamental form is  $\mathrm{I}=(1-f'^2)r'^2 du^2+r^2dv^2$. Solving the ODE $ (1-f'^2)r'^2=\delta r^2$, $\delta=\pm 1$, let $\rho^2=r^2$. Then  
$\mathrm{I}=\rho(u)^2(\delta du^2+dv^2)$.   
\item Spacelike rotation axis. Without loss of generality, we can assume that the rotation axis is the $x_1$-axis. There are two types of parametrizations.
\begin{enumerate}
\item Subcase 1. The parametrization of the surface is 
$$X(r,v)=(r,f(r) \cosh{v},f(r)\sinh{v}),\quad r\in I\subset\r, v\in\r,$$
for some function $f$. Let us replace $r$ by a function $r=r(u)$ to determine.  The first fundamental form is
 $\mathrm{I}=(1+f'^2)r'^2 du^2-f^2 dv^2$, in particular, the surface is timelike. Solving the ODE $ (1+f'^2)r'^2= f^2$ and taking $\rho^2=f^2$, the first fundamental form is $\mathrm{I}=\rho(u)^2(du^2-dv^2)$.   
 \item Subcase 2. The surface is parametrized by 
 $$X(r,v)=(r,f(r) \sinh{v},f(r)\cosh{v}),\quad r\in I\subset\r, v\in\r,$$
 for some function $f$. Again,  let $r=r(u)$ be a function to determine. The first fundamental form is $\mathrm{I}=(1-f'^2)r'^2du^2+f^2dv^2$.   Solving the ODE $(1-f'^2)r'^2=  \delta f^2$, $\delta=\pm 1$, and reparametrizing  the surface, the first fundamental form is
$\mathrm{I}=\rho(u)^2(\delta du^2+dv^2)$.   

\end{enumerate}
\item Lighlike axis. We can assume that the rotation axis is spanned by the vector $(1,0,1)$. The surface is parametrized by
 $$X(r,v)=\left( f(r)+r(1-v^2),-2 v r,f(r)-r(1+v^2) \right),\quad r\in I\subset\r, v\in\r,$$
 for some function $f$. Changing $r$ by $r(u)$, the first fundamental form is
 $\mathrm{I}=(f'r'^2 du^2+r^2 dv^2)$.
  Solving the ODE $  f'r'^2= \delta r^2$, $\delta=\pm 1$, and taking $\rho^2=r^2$, the first fundamental form is
$\mathrm{I}=\rho(u)^2(\delta du^2+dv^2)$.  
\end{enumerate}
\end{proof}

Another example of intrinsic rotational surfaces  without being a surface of revolution is 
the Enneper surface.    In the Lorentz-Minkowski space, there are  two  Enneper surfaces according to their  causal character.  Both surfaces have zero mean curvature and both  have    special expressions  for the Weingarten endomorphism which are now described.

\begin{enumerate}
\item The spacelike Enneper surface \cite{ko}. A parametrization of this surface is
$$(u,v)\longmapsto (u-uv^2+\frac{u^3}{3},-v+u^2 v-\frac{v^3}{3},v^2-u^2), \quad u,v\in\r.$$
With the change of variables $u\to e^u\cos v$ and $v\to e^u\sin v$, a new parametrization is 
\begin{equation}\label{ens}
X(u,v)=e^u\left(\begin{array}{c}
\frac{1}{3}   \left(e^{2 u} \cos (3 v)+3 \cos (v)\right)\\
\frac{1}{3}   \left(e^{2 u} \sin (3 v)-3 \sin (v)\right)\\
-e^{ u} \cos (2 v)\end{array}
\right),\qquad u\not=0.
\end{equation}
At the point $(0,v)$, $X$ is not an immersion. 
The first fundamental form in coordinates with respect to $X$  is
$$\mathrm{I}=e^{2u}(e^{2u}-1)^2\left(\begin{array}{ll} 1&0\\ 0&1\end{array}\right).$$
Thus the surface is intrinsic rotational with $\delta=\epsilon=1$ in \eqref{in1}. To describe the Weingarten endomorphism $A$, we calculate its matrix expression with respect to the basis $\{X_u,X_v\}$. Using that $A=\textrm{I}^{-1} \textrm{II}$, where $\textrm{II}$ is the second fundamental form, a computation gives 
$$A=\frac{2}{(1-e^{2u})^2}\left(\begin{array}{ll}\cos(2v)&-\sin(2v)\\ -\sin(2v)&-\cos (2v)\end{array}\right).$$
This matrix $A$ can be rewritten as
\begin{equation}\label{aa}
 A=R_{-v}\left(\begin{array}{cc}\frac{2}{(e^{2u}-1)^2}&0\\ 0  &-\frac{2}{(e^{2u}-1)^2}\end{array}\right) R_v,
\end{equation}
where 
$$   R_v=\left(\begin{array}{ll}\cos(v)&-\sin(v)\\ \sin(v) &\cos(v)\end{array}\right)$$
   represents the Euclidean rotation of angle $v$ in each tangent plane of the surface. Thus \eqref{aa} can be expressed as  
$$A=\frac{2}{(e^{2u}-1)^2}R_{-v}\left(
\begin{array}{cc}
1 & 0\\
0 & -1\\
\end{array}
\right)R_v,$$
being $\pm 2/(e^{2u}-1)^2$ are the principal curvatures of the surface.
\item The timelike Enneper surface \cite{kon}. A parametrization of this surface  is
$$(u,v)\longmapsto (u^2+v^2,u-\frac{u^3}{3}-u v^2,v+\frac{v^3}{3}+v u^2),\quad u,v\in\r.$$
  With the change $u\to e^u\cosh{v}$ and $v\to e^u\sinh{v}$, a new parametrization of the surface is
\begin{equation}\label{ent}
X(u,v)=e^u\left(\begin{array}{c}
e^{ u} \cosh (2 v)\\
\frac{1}{3} \left(3\cosh(v)-e^{2 u}\cosh(3v)\right)\\
\frac{1}{3}  \left(3\sinh(v)+e^{2 u}\sinh(3v)\right)
\end{array}\right),\qquad u,v\in\r.
\end{equation}
The first fundamental form is
$$\mathrm{I}=e^{2u}(1+e^{2u})^2\left(\begin{array}{ll} 1&0\\ 0&-1\end{array}\right),$$
which shows that the surface is intrinsic rotational with $\delta=-\epsilon= 1$ in \eqref{in1}. 
The matrix expression of $A$ with respect to $\{X_u,X_v\}$ is 
$$A=\frac{2}{(1+e^{2u})^2}\left(\begin{array}{ll}\cosh(2v)&\sinh(2v)\\ -\sinh(2v)&-\cosh (2v)\end{array}\right).$$
We can also write $A$ as
\begin{equation}\label{aaa}
A=G_{-v} \left(\begin{array}{cc}\frac{2}{(e^{2u}+1)^2}&0\\ 0  &-\frac{2}{(e^{2u}+1)^2}\end{array}\right) G_v,  
\end{equation}
where   
$$G_v=\left(\begin{array}{ll}\cosh(v)&\sinh(v)\\ \sinh(v) &\cosh(v)\end{array}\right)$$ is   the hyperbolic rotation of angle $v$ in each tangent plane. Then \eqref{aaa} is 
$$A=\frac{2}{(e^{2u}+1)^2}G_{-v}\left(
\begin{array}{cc}
1 & 0\\
0 & -1\\
\end{array}
\right)G_v.$$
Here $\pm 2/(e^{2u}+1)^2$ are the principal curvatures of the surface.
\end{enumerate}

This particular form of the Weingarten endomorphism in both surfaces also holds for the Euclidean Enneper surface. This fact was   well observed by Freese and Weber  \cite{fw} and inspires  the following definition in Lorentz-Minkowski space, where the angle $v$ in $R_v$ and $G_v$ is replaced by a more general function $\alpha=\alpha(v)$.

\begin{definition}\label{def2}
Let $\Sigma \subset \E_1^3$ be an intrinsic rotational surface with coordinates $(u,v)$ as in \eqref{in1}.   Let $\alpha:\r\rightarrow\r$ be a smooth  function.   Assume that the Weingarten endomorphism $A$ is diagonalizable and let $\lambda_i$ denote the principal curvatures, $i=1,2$. We say that $\Sigma$ has twist $\alpha$ if one of the following two conditions holds:
\begin{enumerate}
\item   $\Sigma $ is spacelike and $A$ is of the form
\begin{equation}\label{ma1}
A=R_{-\alpha(v)}\left(
\begin{array}{cc}
\lambda_1(u) & 0\\
0 & \lambda_2(u)\\
\end{array}
\right)R_{\alpha(v)}. 
\end{equation}
\item   $\Sigma$ is timelike and $A$ is of the form
\begin{equation}\label{ma2}
A=G_{-\alpha(v)}\left(
\begin{array}{cc}
\lambda_1(u) & 0\\
0 & \lambda_2(u)\\
\end{array}
\right)G_{\alpha(v)}.
\end{equation}
 \end{enumerate}
\end{definition}

We make the following observations.
\begin{enumerate}
\item In general,   the Weingarten endomorphism of a timelike surface could not be   diagonalizable. However, in  Definition \ref{def2} it is required that $A$ is diagonalizable for timelike surfaces.  
\item The rotations $R_\theta$ and $G_\theta$ are isometries in each tangent plane. In the first case, $\Sigma$ is spacelike, the tangent planes are Riemannian and $R_\theta$ are   Euclidean rotations.  If $\Sigma$ is timelike, the tangent planes are Lorentzian and $G_\theta$ are   hyperbolic rotations. 
\item  Definition \ref{def2} was motivated by the spacelike and timelike Enneper surfaces. Notice that for the timelike Enneper surface, we have $\delta=-\epsilon=1$. One can consider timelike surfaces satisfying \eqref{ma1} and likewise, the condition \eqref{ma2} can be extended for spacelike surfaces. This makes a remarkable difference between the Lorentzian and the Euclidean contexts. See this discussion in the final part of Section \ref{sec2}.
\end{enumerate}
 
  We shall characterize the intrinsic rotational surfaces with twist $\alpha$ in terms of the mean curvature. 
     This is given in the following result.
  
\begin{theorem}\label{tmain}
Let $\Sigma$ be an  intrinsic rotational surface of $\E_1^3$  with twist  $\alpha\not=0$. Assume   that  $\Sigma$ has no open sets of umbilic points. If $\Sigma$ is spacelike we also assume  that $\alpha$ is not an integer  multiple of $\frac{\pi}{2}$ on any open interval.  Then $\Sigma$ has constant mean curvature and the twist is $\alpha(v)=av+c$, $a,c\in\r$.
\end{theorem}

The proof is done in Section \ref{sec2}. As a consequence of this theorem, for nonzero constant mean curvature cases, a spacelike intrinsic rotational  surface with twist $\alpha$ is one of the surfaces that appeared  in  \cite{brs}. Theorem \ref{tmain} also covers the case that the surface is timelike which does not appear in \cite{brs}. Examples of timelike ZMC surfaces  appear in  Subsection \ref{sec42}. Notice that in timelike surfaces with constant mean curvature the set of umbilic points can contain open sets because we cannot use that this set is the zeroes of a holomorphic function. However, in the statement of Theorem \ref{tmain} we have imposed the condition that umbilic points are isolated.  

Section \ref{sec3}  is devoted to studying the case of intrinsic rotational ZMC  surfaces with twist $\alpha$ obtaining a full classification in Theorems \ref{t33},  \ref{t38} and \ref{t39}. This complements  the work initiated in \cite{brs} for spacelike surfaces and  covering also the case that the surfaces are timelike. These surfaces generalize the two Enneper surfaces of $\E_1^3$, hence that Theorem \ref{tmain}, in the ZMC case, provides a characterization of the Enneper surfaces within the class of intrinsic rotational surfaces.  A goal of this paper is that we provide  new examples of timelike ZMC surfaces. This is due to the fact that in \eqref{in1} we have two choices for the metrics, namely, $\rho(u)^2(du^2-dv^2)$ and $\rho(u)^2(-du^2+dv^2)$. In Subsection \ref{sec5}, it will be proved that the intrinsic rotational ZMC surfaces with constant twist $\alpha$ are associate surfaces of surfaces of revolution (Theorem \ref{t51}). The computations of the Weierstrass representation of the intrinsic rotational minimal surfaces   together with  some examples appear in Section \ref{sec4}.  

As a last remark, in this paper, we only consider that all surfaces are non-degenerate, in particular $\rho\not=0$ in \eqref{in1}. In particular, it will be assumed that the domains of the parametrizations of all surfaces do not include degenerate points.  However, it is natural to investigate if the surface can be extended to points where the metric is lightlike. In fact, given a parametrization of a surface in $\E_1^3$, the surface will contain points of spacelike, lightlike and timelike character.   Recently, there has been considerable interest in studying surfaces with ZMC  having more than one causal character. The literature is    abundant in this matter: see to cite a few \cite{es,fsuy,uy}. As an anonymous referee has kindly pointed out to us, it is an interesting problem to ask for the existence of intrinsic rotational surfaces with lightlike parts, which will include mixed-type surfaces.

\section{Proof of Theorem \ref{tmain}}\label{sec2}

Let $\Sigma$ be a spacelike or timelike surface in $\E_1^3$. If $N$ is a unit normal vector field on $\Sigma$, then   $\langle N,N\rangle=-\sigma$, where $\sigma=1$ if $\Sigma$ is spacelike and $\sigma=-1$ if $\Sigma$ is timelike. If  $\mathfrak{X}(\Sigma)$ denotes the space of tangent vector fields of $\Sigma$,  $\nabla^0$ the Levi-Civita connection of $\E_1^3$ and  $\nabla$ the induced connection on $\Sigma$, then the Gauss formula is
\begin{equation}\label{3-gf}
\nabla_X^0 Y=\nabla_X Y+\textrm{II}(X,Y),\quad X,Y\in \mathfrak{X}(\Sigma).
\end{equation}
   Since  $\nabla^0_X N$ only has a tangent part,    the Weingarten endomorphism $A$ is  defined by  $ AX=-(\nabla_X^0 N)^\top$, where the superscript $^\top$ denotes the tangent part. From  \eqref{3-gf}
\begin{equation}\label{3-gf2}
\langle AX,Y\rangle=\langle \textrm{II}(X,Y),N\rangle.
\end{equation}
 The mean curvature vector field $\vec{H}$  is defined as   $\vec{H}=\frac12\mbox{ trace}(\textrm{II})$ and   the mean curvature function $H$   by the relation $\vec{H}=HN$ \cite{lop}. Since $\textrm{II}(X,Y)$ is proportional to $N$,   from  \eqref{3-gf} and \eqref{3-gf2} we deduce
\begin{equation}\label{s-1}
\textrm{II}(X,Y)=-\sigma \langle \textrm{II}(X,Y),N\rangle N=-\sigma \langle AX,Y\rangle N.
\end{equation}
Because $\textrm{II}$ is symmetric,   $A$  is a self-adjoint endomorphism with respect to the induced metric of $\Sigma$. If $A$ is (real) diagonalizable, the eigenvalues $\lambda_1$ and $\lambda_2$ are called the principal curvatures of $\Sigma$. To be precise, if $\Sigma$ is spacelike, then $A$ is always diagonalizable because the induced metric is Riemannian, but if $\Sigma$ is timelike, $A$ may be not diagonalizable. From now on,  we will assume that $A$ is diagonalizable if $\Sigma$ is timelike.  
 Therefore
$$H=\sigma \langle\vec{H},N\rangle=-\sigma \frac{\lambda_1+\lambda_2}{2}.$$
A surface $\Sigma\subset\E_1^3$ is said to be a {\it minimal surface} if the mean curvature is $H=0$ on $\Sigma$.  In the literature, spacelike minimal surfaces are called maximal surfaces because they   locally maximize the area functional. However,  we will employ the terminology zero mean curvature (ZMC in short) surface independently if $\Sigma$ is spacelike or timelike. 

 The Gauss curvature $K$ of a non-degenerate surface in $\E_1^3$ is $K=-\sigma\mbox{det}(A)$, which writes in terms of the principal curvatures as $K=-\sigma\lambda_1\lambda_2$.

The proof of Theorem \ref{tmain} is done in several steps by means of preliminary lemmas. Let $\Sigma$ be   an intrinsic rotational surface of $\E_1^3$ with coordinates $(u,v)$ determined by \eqref{in1}. Introduce the unitary tangent vector fields
$$
U=\frac{1}{\rho(u)} \partial_u \quad\mbox{and} \quad V=\frac{1}{\rho(u)} \partial_v,
$$
where $\{\partial_u,\partial_v\}$ are the canonical tangent vector fields. In particular,  $\langle U,U\rangle=\delta$ and $\langle V,V\rangle=\epsilon$.

\begin{lemma}\label{le1}
 The Levi-Civita connection  of the first fundamental form $\mathrm{I} $  is given by
\begin{equation}\label{uv}
\nabla_UU=0,\quad   \nabla_UV=0, \quad\nabla_VU=\frac{\rho'}{\rho^2}V, \quad \nabla_VV=-\delta\epsilon\frac{\rho'}{\rho^2}U.
\end{equation}

\end{lemma}

\begin{proof} Any tangent vector field $W\in\mathfrak{X}(\Sigma)$ writes as 
$W=\delta\langle W,U\rangle U+\epsilon\langle W,V\rangle V$. This identity will be used  for the computations of the Levi-Civita connection.  We begin with  $\nabla_UU$. Since $U$ is unitary, $\langle\nabla_UU,U\rangle=0$. We also have, 
$$\langle\nabla_UU,V\rangle=\frac{1}{\rho^3}\langle\nabla_{\partial_u}\partial_u,\partial_v\rangle=-\frac{1}{\rho^3}\langle\nabla_{\partial_v}\partial_u,\partial_u\rangle=-\frac{1}{2\rho^3} \partial_v(\delta\rho^2)=0$$
because $\rho$ depends only on $u$. This proves $\nabla_UU=0$. In fact, this means that $U$ is a Killing vector field on $\Sigma$.

 We now compute  $\nabla_UV$. Note that $\langle\nabla_UV,U\rangle=-\langle\nabla_UU,V\rangle=0$. Also, $\langle \nabla_UV,V\rangle=\frac12U\langle V,V\rangle=0$. Hence, $\nabla_UV=0$.

For the calculation of $\nabla_VU$, we have $\langle \nabla_VU,U\rangle=\frac12V\langle U,U\rangle=0$. On the other hand,
$$
\langle \nabla_VU,V\rangle=\frac{1}{\rho^3}\langle\nabla_{\partial_v}\partial_u,\partial_v\rangle=\frac{1}{2\rho^3}\partial_u(\epsilon\rho^2)=\frac{\epsilon\rho'}{\rho^2}.$$
Thus $\nabla_VU=\epsilon \langle \nabla_VU,V\rangle V=\rho'/\rho^2 V$.

The computation of vector field $\nabla_VV$ is similar.   Since $V$ is unitary, we have $\langle \nabla_VV,V\rangle=0$. Furthermore,
$$\langle \nabla_VV,U\rangle=-\langle V,\nabla_VU\rangle=-\epsilon \frac{\rho'}{\rho^2}.$$

 \end{proof}

\begin{lemma} The Gauss equation is equivalent to
\begin{equation}\label{first}
\lambda_1 \lambda_2 =  \epsilon\frac{ \rho \rho'' -\rho'^2}{\rho^4 }.
\end{equation}
\end{lemma}
\begin{proof}
The   curvature tensor $R$ is $R(U,V)V=\nabla_U\nabla_VV-\nabla_V\nabla_UV-\nabla_{[U,V]}V$. 
By using the identities \eqref{uv} and also $[U,V]=-\nabla_VU=-\frac{\rho'}{\rho^2}V$, we have
$$R(U,V)V=\nabla_U\left(-\delta\epsilon\frac{\rho'}{\rho^2}U\right)+\frac{\rho'}{\rho^2}\nabla_VV=-\frac{\delta\epsilon}{\rho^4}(\rho\rho''-\rho'^2)U.$$
Thus
$$\langle R(U,V)V,U\rangle=-\frac{\epsilon}{\rho^4}(\rho\rho''-\rho'^2).$$
On the other hand, $\langle U,U\rangle\langle V,V\rangle-\langle U,V\rangle^2=\delta\epsilon$.
Then
$$K=\frac{\langle R(U,V)V,U\rangle}{\langle U,U\rangle\langle V,V\rangle-\langle U,V\rangle^2}=-\frac{\delta}{\rho^4}(\rho\rho''-\rho'^2).$$
Taking into account that $\sigma=\delta\epsilon$, the above identity in combination with $K=-\sigma\lambda_1\lambda_2$ yields \eqref{first}.
\end{proof}

\begin{lemma} Let $\Sigma$ be an intrinsic rotational surface in $\E_1^3$ with twist $\alpha$. Then
the Codazzi equations are equivalent to:
\begin{enumerate}
\item If $\Sigma$ is spacelike, then
\begin{equation}\label{cod11}
\left\{\begin{split}
\sin(2\alpha)\rho(\lambda_1'+\lambda_2')&=0\\
-\lambda_1'\sin^2(\alpha)+\lambda_2'\cos^2(\alpha)+\frac{1}{\rho}(\lambda_1-\lambda_2)(\rho\alpha'-\rho')&=0.
\end{split}\right.
\end{equation}
\item  If $\Sigma$ is timelike, then 
\begin{equation}\label{cod12}
\left\{\begin{split}
\sinh(2\alpha)\rho(\lambda_1'+\lambda_2')&=0\\
 \sinh^2(\alpha)\lambda_1'+\cosh^2(\alpha)\lambda_2'+\frac{1}{\rho}(\lambda_1-\lambda_2)(\rho\alpha'-\rho')&=0.
\end{split}\right.
\end{equation}
 \end{enumerate}
\end{lemma}

\begin{proof}
The Codazzi equations are obtained by the identity $(\nabla_U A)V=(\nabla_V A)U$. Let $A=(a_{ij})$ be the matrix expression of $A$ with respect to $\{\partial_u,\partial_v\}$. Since $\{U,V\}$ are proportional to $\partial_u$ and $\partial_v$, we have
\begin{equation}\label{aij}
AU=a_{11}U+a_{21}V,\quad AV=a_{12}U+a_{22}V.
\end{equation}
Using \eqref{uv} and \eqref{aij}, we have
\begin{eqnarray*}
(\nabla_U A)V&=&\nabla_U (AV)-A(\nabla_U V)=\nabla_U(a_{12}U+a_{22}V)=U(a_{12})U+U(a_{22})V\\
&=&\frac{1}{\rho}\left((a_{12})_uU+(a_{22})_uV\right).
\end{eqnarray*}
\begin{eqnarray*}
(\nabla_V A)U&=&\nabla_V (AU)-A(\nabla_V U)=V(a_{11})U+V(a_{21})V+a_{11}\nabla_VU+a_{21}\nabla_VV-A(\nabla_VU)\\
&=&\frac{1}{\rho}\left((a_{11})_vU+(a_{21})_vV\right)+a_{11}\nabla_VU+a_{21}\nabla_VV-A(\nabla_VU)\\
&=&\left(\frac{(a_{11})_v}{\rho}-\frac{\rho'}{\rho^2}(\delta\epsilon a_{21}+a_{12})\right)U+
\left(\frac{(a_{21})_v}{\rho}+\frac{\rho'}{\rho^2}(a_{11}-a_{22})\right)V.
\end{eqnarray*}
Equating coordinate-by-coordinate,  the Codazzi equations are
\begin{equation}\label{wor}
\left\{\begin{split}
\rho\left((a_{12})_u-(a_{11})_v\right)+\rho'(\delta\epsilon a_{21}+a_{12})=&0\\
\rho\left((a_{22})_u-(a_{21})_v\right)-\rho'(a_{11}-a_{22})=&0.
\end{split}\right.
\end{equation}
We particularize both equations for spacelike and timelike surfaces. 
\begin{enumerate}
\item If $\Sigma$ is spacelike, then the matrix $A$ in \eqref{ma1} is
\begin{equation*}\label{w1}
A=\left(
\begin{array}{cc}
   \cos ^2(\alpha)\lambda_1+\sin ^2(\alpha) \lambda_2 &  \sin(2\alpha)\frac{\lambda_2-\lambda_1}{2} \\
  \sin(2\alpha)\frac{\lambda_2-\lambda_1}{2} &  \cos ^2(\alpha)\lambda_2+\sin ^2(\alpha) \lambda_1 \\
\end{array}
\right).
\end{equation*}
Using this expression for $A$,   together \eqref{wor} with $ \delta=\epsilon=1$, the Codazzi equations are
\begin{equation*}
\left\{
\begin{split}
\sin (\alpha) \cos (\alpha) \left(\rho \left(2 \alpha '(\lambda_1-\lambda_2)-\lambda_1'+\lambda_2'  \right)-2 \rho ' (\lambda_1-\lambda_2) \right)&=0\\
 \rho  \left((\lambda_1-\lambda_2) \alpha ' \cos (2 \alpha )+\lambda_1'\sin ^2(\alpha )+\lambda_2' \cos ^2(\alpha )\right)+(\lambda_2-\lambda_1) \rho ' \cos (2 \alpha )&=0.
\end{split} \right.\end{equation*}
A linear combination of both equations gives \eqref{cod11}.
\item  If $\Sigma$ is timelike, then the matrix $A$ in   \eqref{ma2} is
\begin{equation*}\label{w2}
A=\left(
\begin{array}{cc}
 \cosh ^2(\alpha)  \lambda_1-\sinh ^2(\alpha) \lambda_2 & -\sinh (2\alpha) \frac{\lambda_2-\lambda_1}{2}\\
   \sinh (2\alpha) \frac{\lambda_2-\lambda_1}{2} & \cosh ^2(\alpha) \lambda_2-\sinh ^2(\alpha) \lambda_1 \\
\end{array}
\right).
\end{equation*}
Using these values for $a_{ij}$ in the Codazzi equations \eqref{wor}, and taking into account that $\sigma=-1$, we obtain \eqref{cod12}.
  
\end{enumerate}
\end{proof}

After all these preliminaries, we prove Theorem \ref{tmain}.

\begin{proof} ({\it of Theorem \ref{tmain}.})

We distinguish the cases when $\Sigma$ is spacelike or timelike.
\begin{enumerate}
\item  $\Sigma$ is spacelike.   Since $\alpha$ is not an integer multiple of $\pi/2$,  $\sin(2\alpha)\not=0$, the first equation of \eqref{cod11} implies that $H$ is constant. In particular, $\lambda_2'=-\lambda_1'$. Now  the second equation of \eqref{cod11} is
$$\rho\lambda_1'=(\lambda_1-\lambda_2)(\rho\alpha'-\rho').$$
Since $\alpha$ depends only on $v$, and $\rho$ and $\lambda_i$ depend on $u$, we deduce that $\alpha'$ is constant or $\lambda_1-\lambda_2=0$ on $\Sigma$. The latter case is not possible because there are no  open sets of umbilic points. This proves that $\alpha'$ is constant, so $\alpha$ is linear.
 
\item  $\Sigma$ is timelike.   The first equation of \eqref{cod12} implies that $H$ is constant. A similar discussion as in the above case (1) using  the second equation of \eqref{cod12} concludes that the function $\alpha$ is linear.
 
\end{enumerate}

\end{proof}

\begin{remark} The hypothesis $\alpha\not=0$ in Theorem \ref{tmain} has been employed to deduce from the first equation in the Codazzi equation that the mean curvature is constant. In  Section \ref{sec5}, we will study  the case  $\alpha=0$ for ZMC surfaces, obtaining that the surface is a surface of revolution (Theorem \ref{t51}).
\end{remark}

\begin{corollary}  \label{coh}
Assume that $\Sigma$ has the same hypothesis of Theorem \ref{tmain}, with $\alpha(v)=av+c$, $a,c\in\r$. Then the principal curvature $\lambda_1$ is
\begin{equation}\label{h1}
\lambda_1=-\sigma H+\frac{b}{\rho^2}e^{2au},\quad b\in\r, 
\end{equation}
and the function $\rho$ satisfies 
  \begin{equation}\label{difeq1}
 \rho\rho''-\rho'^2=\epsilon( H^2\rho^4-b^2e^{4au}).
 \end{equation}
 In addition, the constant $b$ cannot be $0$.  
\end{corollary}

\begin{proof} 
Using that the mean curvature is constant, that is, $\lambda_2'=-\lambda_1'$, the second Codazzi equations in \eqref{cod11} and \eqref{cod12} write in both cases as  $-\rho\lambda_1'+2(\lambda_1+\sigma H)(a \rho-\rho')=0$. Then 
$$\frac{\lambda_1'}{\lambda_1+\sigma H}=2\left(a-\frac{\rho'}{\rho}\right).$$
A straightforward integration of this equation yields \eqref{h1}. Equation \eqref{difeq1} is immediate from \eqref{first} and \eqref{h1}. 

If $b=0$, then $\lambda_1=-\sigma H$ by \eqref{h1}. Thus $\lambda_2=-\sigma H$ by the definition of $H$. This is  a contradiction because there are no open sets of umbilic points.  
 \end{proof}

\begin{remark} Following a suggestion of one of the anonymous referees, if we admit to define intrinsic rotational surfaces with lightlike parts, and assuming that the mean curvature is constant, then we can admit $\rho(u_0)=0$ because the solution of   \eqref{difeq1} sometimes become $\rho(u_0)=0$ real analytically.
\end{remark}

As we have observed in Section \ref{sec1}, the conditions \eqref{ma1} and \eqref{ma2} can be extended for timelike and spacelike surfaces respectively. We study this situation which has no  counterpart in the Euclidean space. First, we recall that the skew curvature of a surface is defined as $\sqrt{H^2+\sigma K}$: see \cite{da} in the Lorentzian context and \cite{lp,tp} in the Euclidean one. If $A$ is diagonalizable, the skew curvature is, up to a change of sign, the difference $\lambda_2-\lambda_1$ between the principal curvatures.

\begin{theorem} Let $\Sigma$ be an  intrinsic rotational surface in $\E_1^3$  which has no open sets of umbilic points.  Assume that  $\Sigma$ is a timelike surface and satisfies \eqref{ma1} where $\alpha$ is not an integer multiple of $\pi/2$ or that $\Sigma$ is spacelike and satisfies  \eqref{ma2} with $\alpha\not=0$. Then the function $\alpha$ is constant. Furthermore, $\Sigma$ has constant skew curvature or $\Sigma$ is a cylinder.
\end{theorem}
\begin{proof}
\begin{enumerate}
\item Suppose that  $\Sigma$ is a timelike intrinsic rotational  surface satisfying \eqref{ma1}. We have two cases to discuss, namely,   $\delta=-\epsilon=1$ and $\delta=-\epsilon=-1$. The arguments are similar in both cases, so without loss of generality, we will assume $\delta=-\epsilon=1$. Computing again  the Codazzi equations \eqref{wor}, we obtain
\begin{equation}\label{cod112}
\left\{
\begin{split}
 \sin(2\alpha)\rho\left(\lambda_2'-\lambda_1'+2\alpha'( \lambda_1-\lambda_2)\right)&=0\\
\lambda_1'\sin^2(\alpha)+\lambda_2'\cos^2(\alpha)+\frac{\lambda_1-\lambda_2}{\rho}\cos(2\alpha)(\rho\alpha'-\rho')&=0.
\end{split}\right.
\end{equation}
Since $\sin(2\alpha)\not=0$, the first equation of \eqref{cod112} gives $2\alpha'(\lambda_1-\lambda_2)=\lambda_1'-\lambda_2'$. Substituting into the second one,   
$$\lambda_1'+\lambda_2'-2\cos(2\alpha)(\lambda_1-\lambda_2)\frac{\rho'}{\rho}=0.$$
Because $\alpha$ depends on $v$ and $\lambda_i$ and $\rho$ depend on $u$,    this  equation implies that $\alpha$ is constant or $\rho'=0$.
\begin{enumerate}
\item Case $\alpha$ is constant. The first equation of \eqref{cod112} gives $\lambda_2-\lambda_1=a$, $a\in\r$, $a\not=0$. Thus the skew curvature is constant. 
\item Case $\rho'=0$.   From \eqref{first}, $\lambda_1\lambda_2=0$. Without loss of generality, suppose  $\lambda_1=0$ on $\Sigma$. Now the second equation of \eqref{cod112} is $\lambda_2'=0$, proving that $\lambda_2$ is constant. Since both principal curvatures are constant, the surface is isoparametric with $K=0$.   Because $\Sigma$ has no open sets of umbilic points, then $\Sigma$ is   a cylinder \cite{lo}. In particular, $\lambda_2\not=0$. Finally the first equation of \eqref{cod112} gives $\alpha'=0$, hence $\alpha$ is constant.
 
\end{enumerate}

\item Suppose that $\Sigma$ is a spacelike intrinsic rotational  surface satisfying \eqref{ma2}. The Codazzi equations \eqref{wor} are now
\begin{equation}\label{cod122}
\left\{
\begin{split}
\cosh(\alpha)\sinh(\alpha)\rho\left( 2\alpha'(\lambda_2-\lambda_1)+\lambda_1'-\lambda_2'\right) &=0\\
 \rho(-\sinh^2(\alpha)\lambda_1'+\cosh^2(\alpha)\lambda_2')+\cosh(2\alpha)(\lambda_1-\lambda_2)(\rho\alpha'-\rho')&=0.
\end{split}\right.
\end{equation}
From the first equation of \eqref{cod122}, we have $2\alpha'(\lambda_1-\lambda_2)=\lambda_1'-\lambda_2'$. Substituting into the second equation, 
$$ \lambda_1'+\lambda_2'-2\cosh(2\alpha)(\lambda_1-\lambda_2)\frac{\rho'}{\rho}=0.$$
 Since $\alpha$ only depends on $v$ and $\lambda_i$ and $\rho$ depend on $u$, we deduce that $\alpha$ is constant or $\rho'=0$ because  there are no open sets of umbilic points. The discussion is now similar to the previous case when $\Sigma$ is timelike, obtaining the result. 
\end{enumerate}\end{proof}


\section{Intrinsic rotational surfaces with zero mean curvature}\label{sec3}

In this section, we investigate the intrinsic rotational  surfaces with twist $\alpha$ and zero mean curvature.  By Theorem \ref{tmain}, the twist  is $\alpha(v)=av+c$, $a,c\in\r$. A first step consists in proving that the study can be reduced to the case $c=0$.   This is a consequence of a more general result of spacelike and timelike ZMC surfaces involving the notion of associate surfaces. We distinguish if the surface is spacelike or timelike. 

Suppose that   $X\colon \Sigma\to\E_1^3$ is a  conformal spacelike ZMC immersion. Since $X$ is conformal and $H=0$,  $X$ is  harmonic.  As in the Euclidean space  (\cite[Ch. 3]{di}), 
the associate ZMC surface $X^\theta$, $\theta\in\r$, is defined by $X^\theta=\cos\theta X+\sin\theta X^*$,    where $X^*$ is the conjugate surface of $X$. The surface $X^*$ is defined by the property that $X+iX^*$ is holomorphic. In particular,   $X_u^*=X_v$ and $X_v^*=-X_u$, hence the tangent planes of $X$ and $X^{\theta}$ coincide at corresponding points. Moreover, the conjugate surface $X^*$ is just $X^{\pi/2}$. The associate surface $X^\theta$ is isometric to $X$ and shares the same Gauss map, $N^\theta\circ X^\theta=N\circ X$. The following result gives the relation  of the corresponding endomorphisms, which    is known as the Bonnet-Lie transformation \cite[\S 394]{bi} and \cite{ak}. 

\begin{proposition}\label{pr-s}  If $X\colon \Sigma\to\E_1^3$ is a  conformal spacelike ZMC immersion, then the relation between the Weingarten endomorphisms $A$ and $A^\theta$ of $X$ and $X^\theta$, respectively, is 
\begin{equation}\label{ara}
A^\theta\circ X^\theta=R_\theta A R_{-\theta}\circ X.
\end{equation}
\end{proposition}
 
This proposition has an analog for timelike surfaces. In the theory of   timelike ZMC surfaces, the complex analysis is replaced by the paracomplex analysis \cite{er,kon}.  In particular, it is possible to define associate surfaces for a given  timelike ZMC surface in a similar way with spacelike surfaces. Although this concept is expectable, as far as the authors know, it is not clearly explicit in the literature. We include a short review for sake of completeness.  Let $\l$ be the algebra of paracomplex numbers, $z=u+\tau v$, where $\tau$ is the pure paracomplex number with $\tau^2=1$. Introduce the operators
$$\frac{\partial}{\partial z}=\frac12\left(\frac{\partial}{\partial_u}+\tau\frac{\partial}{\partial_v}\right),\quad 
\frac{\partial}{\partial \bar{z}}=\frac12\left(\frac{\partial}{\partial_u}-\tau\frac{\partial}{\partial_v}\right).$$
Let $X\colon \Sigma \to\E_1^3$ be a paraconformal timelike surface, $X=X(u,v)$, where the metric is $\rho(z)^2(du^2-dv^2)$ and $\rho(z)$ is a paraholomorphic function.  Then $X$ has zero mean curvature if and only if $\frac{\partial X}{\partial z}$ is paraholomorphic, that is, 
$\frac{\partial}{\partial\bar{z}}(\frac{\partial X}{\partial z})=0$. As in the Euclidean case, there is a paraconjugate map $X^*\colon \Sigma\to\E_1^3$ such that $X+\tau X^*$ is paraholomorphic. In consequence, $X_u=X^{*}_v$ and $X_v=X_u^{*}$. The {\it associate surface} of $X$ with angle $\theta$ is the immersion $X^\theta$ defined by  
$$X^\theta=\cosh\theta X+\sinh\theta X^*,$$
where $\theta\in\r$. It is immediate the following result.

\begin{proposition} Let $X$ be a   paraconformal timelike ZMC surface of $\E_1^3$. Then:
\begin{enumerate}
\item The tangent planes of $X$ and $X^\theta$ coincide at corresponding points and  also $X$ and $X^\theta$ are isometric.
\item The associate surface $X^\theta$ is a  timelike ZMC surface.
\item The Gauss maps of $X$ and $X^\theta$ coincide, $N^\theta\circ X^\theta=N\circ X$.
\end{enumerate}
\end{proposition}

Similarly as in Proposition \ref{pr-s}, we have the relation between the Weingarten endomorphisms \cite{it,mi,we}.

\begin{proposition}\label{pr-t}  If $X\colon \Sigma\to\E_1^3$ is a   paraconformal timelike ZMC surface, then the relation between the Weingarten endomorphisms $A$ and $A^\theta$ of $X$ and $X^\theta$, respectively, is
\begin{equation}\label{aga}
A^\theta\circ X^\theta=G_\theta A G_{-\theta}\circ X.
\end{equation}
\end{proposition}
 
We go back to Theorem \ref{tmain}. Let $\Sigma$ be an intrinsic rotational surface with twist $\alpha(v)=av+c$. If $\Sigma$ has  zero mean curvature, and thanks to Propositions \ref{pr-s} and \ref{pr-t}, we can assume, up to associate surfaces, that the constant $c$ is $0$. From \eqref{difeq1}, the conformal factor $\rho$ satisfies 
\begin{equation}\label{r3}
\rho\rho''-\rho'^2+\epsilon  b^2 e^{4a u}=0.
\end{equation}
Furthermore, by Corollary \ref{coh}, the constant $b$ cannot be $0$.  Let us replace $\rho$ with a homothetic metric $\lambda\rho$, $\lambda>0$, which it does not change the property of being  intrinsic rotational surface with twist. Then equation \eqref{r3} is 
$$\rho\rho''-\rho'^2+\epsilon  \frac{b^2}{\lambda^2} e^{4a u}=0.$$
By taking $\lambda^2=b^2$,  without loss of generality we can assume $b=1$.

\begin{lemma}
All positive solutions of the differential equation
\begin{equation}\label{r4}
\rho\rho''-\rho'^2+ \epsilon  e^{4a u}=0 
\end{equation}
are given by
\begin{equation}\label{r1}
\rho(u)=\frac{e^{2au}}{2B}\left(-\epsilon  Ae^{Bu}+\frac{1}{A}e^{-Bu}\right)
\end{equation}
 for arbitrary $A,B>0$.
\end{lemma}

\begin{proof}
The ODE  \eqref{r4} is similar with Equation (6) in \cite{fw} and the proof follows the same steps. 
\end{proof}

We will obtain the explicit parametrizations $X=X(u,v)$ of all intrinsic rotational ZMC surfaces. The methods we use are moving frames and the Frobenius integrability theorem. From \eqref{r1} we have explicitly the metric $\rho^2\mathrm{I}$ and the second fundamental form is determined in \eqref{h1}. Then  we can start to derive the parametrization $X=X(u,v)$ of the surface directly from the proof of the Bonnet's theorem: see \cite{kta,te}.

A first step in the integration process consists in obtaining an explicit curve $c(u)$  on the surface and an orthogonal frame of vector fields along $c(u)$, where one of the vector fields is tangent along $c(u)$. The final part of the derivation of $X(u,v)$ will use the solution of the Bj\"{o}rling problem  for spacelike and  timelike ZMC surfaces \cite{acm,cha}.

  Let $\{E_1(u,v),E_2(u,v),E_3(u,v)\}$ be an orthonormal basis   determined by the conditions
$$  E_1=\frac{1}{\rho}X_u,\quad   E_2=\frac{1}{\rho}X_v,\quad E_3=E_1\times E_2,$$
where $\times $ is the cross product in $\E_1^3$. Evaluating at $(u,0)$,   we have
\begin{equation*}
\left\{\begin{split}
E_1(u,0)&=dX_{(u,0)}U(u,0)=\frac{1}{\rho} X_u(u,0)\\
 E_2(u,0)&=dX_{(u,0)}V(u,0)=\frac{1}{\rho} X_v(u,0)\\
 E_3(u,0)&=E_1(u,0)\times E_2(u,0).
 \end{split}\right.
 \end{equation*}
  Let $\mathcal{F}(u,v)=(E_1(u,v),E_2(u,v),E_3(u,v))$. Then the Gauss-Codazzi equations are written as
\begin{equation}\label{cc}
\left\{
\begin{split}
\mathcal{F}_u(u,0)&=\mathcal{F}(u,0)P\\
\mathcal{F}_v(u,0)&=\mathcal{F}(u,0)Q,
\end{split}\right.
\end{equation}
where $P$ and $Q$ are two   matrices of order $3$.  More precisely, if  $P=(p_{ij})$ and $Q=(q_{ij})$, then $p_{ij}$ and $q_{ij}$ are determined by
$$p_{ij}=\langle (E_j)_u,E_i\rangle,\quad q_{ij}=\langle (E_j)_v,E_i\rangle.$$

\subsection{Spacelike case}\label{sub-s}

We know that  $E_3$ is timelike. Since $\{E_1(u,0),E_2(u,0),E_3(u,0)\}$ is an orthonormal basis and $\rho$ depends only on the variable $u$, the matrices $P$ and $Q$ are
\begin{equation}\label{pq}
P=\left(\begin{array}{ccc}
0&0&-\frac{e}{\rho}\\
0&0&-\frac{f}{\rho}\\
-\frac{e}{\rho}&-\frac{f}{\rho}&0\end{array}\right),\quad
Q=\left(\begin{array}{ccc}
0&-\frac{\rho'}{\rho}&-\frac{f}{\rho}\\
\frac{\rho'}{\rho}&0&-\frac{g}{\rho}\\
-\frac{f}{\rho}&-\frac{g}{\rho}&0\end{array}\right),
\end{equation}
which all are evaluated at $(u,0)$. Notice that in the Euclidean case, both matrices are skew symmetric which here it is not possible because $E_3$ is timelike. To calculate the coefficients $e$, $f$ and $g$ of the second fundamental form, we use  $A=\textrm{I}^{-1}\textrm{II}$ and \eqref{ma1}. Then 
$$\left(\begin{array}{ll} e&f\\ f&g\end{array}\right)=\rho^2\left(\begin{array}{ll}
\lambda_1\cos^2(av)+\lambda_2\sin^2(av)& (\lambda_2-\lambda_1)\sin(av)\cos(av)\\
(\lambda_2-\lambda_1)\sin(av)\cos(av)&\lambda_1\sin^2(av)+\lambda_2\cos^2(av)\end{array}\right).$$
Evaluating at $(u,0)$,  
$$\left(\begin{array}{ll} e&f\\ f&g\end{array}\right)(u,0)=\rho^2\left(\begin{array}{ll}
\lambda_1& 0\\
0&\lambda_2\end{array}\right)(u,0).$$
From \eqref{cc} and \eqref{pq},
\begin{equation*} 
\left\{
\begin{split}
(E_1)_u(u,0)&=-\rho\lambda_1 E_3(u,0)\\
(E_2)_u(u,0)&=0\\
(E_3)_u(u,0)&=-\rho\lambda_1 E_1(u,0).
\end{split}\right.\end{equation*}
We are able to integrate these equations. From now on, we drop the notation $()_{u}$ and $(u,0)$ employing $(')$ and $(u)$, respectively. Using \eqref{h1}  and \eqref{r1}, we have
\begin{equation}\label{eee}
\left\{
\begin{split}
E_1'(u)&=-\frac{e^{2au}}{\rho} E_3(u)=-\frac{2AB e^{Bu}}{1-A^2e^{2Bu}} E_3(u)\\
E_2'(u)&=0\\
E_3'(u)&=-\frac{e^{2au}}{\rho} E_1(u)=-\frac{2ABe^{Bu}}{1-A^2e^{2Bu}} E_1(u).
\end{split}\right.
\end{equation}
In particular, $E_2$ is constant.  With respect to a fixed orthonormal basis of $\E_1^3$, we can write
$E_1(u)=(a(u),0,b(u))$, $E_2(u)=(0,1,0)$ and $E_3(u)=(-b(u),0,-a(u))$. Letting 
$$h(u)=\frac{2ABe^{Bu}}{1-A^2e^{2Bu}},$$
 the  two differential equations \eqref{eee} involving $E_1$ and $E_3$ are now 
\begin{equation*}
\begin{split}
(a',0,b')&=h(b,0,a)\\
 (b',0,a')&=h(a,0,b).
 \end{split}\end{equation*}
Using that $|E_1|^2=-|E_3|^2=1$, we have $a'^2-b'^2=-h^2$. Hence the function $h$ satisfies
$$\frac{a'}{\sqrt{a^2-1}}=h.$$
The solution $a(u)$ of  this equation is 
$$a(u)= \frac{1+A^2 e^{2Bu}}{1-A^2 e^{2Bu}}+d,\quad d\in\r.$$
 For $b$, we use that $b=a'/h$, to obtain 
$$ b(u)= \frac{2A e^{Bu}}{1-A^2 e^{2Bu}}.$$
With the functions $a$ and $b$, the frame $\{E_1(u),E_2(u),E_3(u)\}$ is completely determined,      
\begin{equation}\label{normal}
\left\{
\begin{split}
E_1(u)&=\frac{1}{1-A^2 e^{2Bu}}\left(1+A^2 e^{2Bu},0,2A e^{Bu}\right)\\
E_2(u)&=(0,1,0)\\
 E_3(u)&=-\frac{1}{1-A^2 e^{2Bu}}\left(2A e^{Bu},0,1+A^2 e^{2Bu}\right).
 \end{split}
 \right.
 \end{equation}
The curve $c(u)=X(u,0)$ is obtained from  $c'(u)=\rho(u) E_1(u)$. This equation can be solved by using \eqref{r1} and the value $E_1(u)$ in \eqref{normal}, to obtain  an explicit parametrization of $c(u)$. 
\begin{lemma}\label{l35}
 The space curve $c(u)=X(u,0)$   is given by
\begin{eqnarray*}
c(u)&=&\frac{e^{2au}}{2B}\left(\frac{Ae^{Bu}}{2a+B}+\frac{e^{-Bu}}{A(2a-B)},0,\frac{1}{a}\right), \quad B\neq\pm2a\\
c(u)&=&\frac{e^{2au}}{4a^2}\left(\frac{Ae^{2au}}{4},0,1\right)+\frac{u}{4Aa}\left(1,0,0\right), \quad B=2a\\
c(u)&=&-\frac{e^{2au}}{4a^2A}\left(\frac{e^{2au}}{4},0,1\right)-\frac{Au}{4a}\left(1,0,0\right), \quad B=-2a.
\end{eqnarray*}
\end{lemma}
Until here, we have a curve $c(u)$ on the surface $X(u,v)$ with $X(u,0)=c(u)$ and an orthonormal frame $\{E_1(u),E_2(u),E_3(u)\}$ along $c$ given in \eqref{normal}. Instead of following with the Frobenius theorem to get the parametrization $X(u,v)$ of the surface, we apply the Bj\"{o}rling formula. From \eqref{normal}, the unit normal vector field of the surface along  the curve $c(u)$ is the vector field $E_3(u)$. Recall that the Bj\"{o}rling formula requires that the initial data $c(u)$ and $E_3(u)$ have analytic extensions which  is clear from \eqref{normal} and Lemma \ref{l35}.  Then the parametrization of the surface is given by the Bj\"{o}rling formula 
\begin{equation}\label{bs}
X(u,v)=\mbox{Re}\left(c(z)+i \int^z E_3(w)\times c'(w)\, dw\right).
\end{equation}
See \cite{acm}. A computation of the integral and real part gives  the explicit parametrization of the surface. We give the first fundamental form and the domain of the surface.

\begin{theorem} \label{t33}
Any spacelike intrinsic rotational  ZMC surface of $\E_1^3$ with twist $\alpha(v)=av$ is parametrized by
\begin{enumerate}
\item Case $B\not=\pm 2a$,
$$X(u,v)=\left(\begin{array}{c}
\frac{e^{(2 a -B) u} \left(A^2 (2 a-B) e^{2 B u} \cos ( (2 a+B)v)+(2 a+B) \cos ( (2 a-B)v)\right)}{2AB(4 a^2 -  B^2)}\\
\frac{e^{(2 a -B) u} \left((2 a+B) \sin (  (2 a-B)v)-A^2 (2 a-B) e^{2 B u} \sin (  (2 a+B)v)\right)}{2AB(4 a^2 -  B^2)}\\
\frac{e^{2 a u} \cos (2 a v)}{2 a B}
\end{array}\right).$$
The first fundamental form is 
$$\textrm{I}=\frac{e^{(4a-2B)u}(A^2e^{2Bu}-1)^2}{4A^2B^2}(du^2+dv^2),$$
and the domain of $X$ is $u\not=-\frac{\log{A}}{B}$, $v\in\r$.
\item Case $B=2a$,
 $$X(u,v)=\left(
 \frac{A^2 e^{4 a u} \cos (4 a v)+4 a u}{16 a^2 A},  \frac{4 a v-A^2 e^{4 a u} \sin (4 a v)}{16 a^2 A},
 \frac{e^{2 a u} \cos (2 a v)}{4 a^2}
 \right).$$
The first fundamental form is 
$$\textrm{I}=\frac{ (A^2e^{4au}-1)^2}{16a^2A^2}(du^2+dv^2),$$
and the domain of $X$ is $u\not=-\frac{\log{A}}{2a}$, $v\in\r$.
\item Case $B=-2a$,
 $$X(u,v)=\left(
 -\frac{e^{4 a u} \cos (4 a v)}{16 a^2 A}-\frac{A u}{4 a},\frac{A v}{4 a}-\frac{e^{4 a u} \sin (4 a v)}{16 a^2 A},-\frac{e^{2 a u} \cos (2 a v)}{4 a^2}
 \right).$$
 The first fundamental form is 
$$\textrm{I}=\frac{ (A^2-e^{4au})^2}{16a^2A^2}(du^2+dv^2),$$
and the domain of $X$ is $u\not=\frac{\log{A}}{2a}$, $v\in\r$.
   \end{enumerate}
\end{theorem}

At the points where $X$ is not defined, the mapping $X$ is not an immersion. The case $A=B=a=1$ corresponds to the spacelike Enneper surface \eqref{ens}. The cases $B=\pm 2a$ correspond with spacelike ZMC surfaces where the $1$-forms in the Weierstrass representation have residues. This implies that the surface is periodic: see this discussion in Section \ref{sec4}.

\subsection{Timelike case}\label{sub-t}

Let   $\Sigma$ be a timelike intrinsic rotational  ZMC surface and twist $\alpha(v)=av$.
Now $\{E_1(u,0),E_2(u,0),E_3(u,0)\}$ is an orthonormal basis where $E_3$ is spacelike. We have two cases to discuss, namely, $\delta=-\epsilon=1$ and $\delta=-\epsilon=-1$. Equivalently, if $E_1$ is spacelike and $E_2$ is timelike or $E_1$ is timelike and $E_2$ is spacelike, respectively. Since the arguments are similar, we give the details in the first case, and we omit the details in the second one, indicating only the results. 

\subsubsection{Case   $E_1$ is spacelike and $E_2$ is timelike}

 In such a case, the matrices $P$ and $Q$ given in \eqref{cc}  are  now
\begin{equation}\label{pq-t}
P=\left(\begin{array}{ccc}
0&0&-\frac{e}{\rho}\\
0&0&\frac{f}{\rho}\\
\frac{e}{\rho}&\frac{f}{\rho}&0\end{array}\right),\quad
Q=\left(\begin{array}{ccc}
0&\frac{\rho'}{\rho}&\frac{f}{\rho}\\
\frac{\rho'}{\rho}&0&\frac{g}{\rho}\\
\frac{f}{\rho}&\frac{g}{\rho}&0\end{array}\right).
\end{equation}
 Using \eqref{ma2} and $A=\textrm{I}^{-1}\textrm{II}$, 
$$\left(\begin{array}{ll} e&f\\ f&g\end{array}\right)=\rho^2\left(\begin{array}{ll}
\lambda_1\cosh^2(av)-\lambda_2\sinh^2(av)& (\lambda_1-\lambda_2)\sinh(av)\cosh(av)\\
(\lambda_1-\lambda_2)\sinh(av)\cosh(av)&\lambda_1\sinh^2(av)-\lambda_2\cosh^2(av)\end{array}\right).$$
Evaluating at $(u,0)$, 
$$\left(\begin{array}{ll} e&f\\ f&g\end{array}\right)(u,0)=\rho^2\left(\begin{array}{ll}
\lambda_1& 0\\
0&-\lambda_1\end{array}\right)(u,0).$$
From \eqref{cc} and \eqref{pq-t} and with the same notation as in the previous subsection, we have
\begin{equation*}
\left\{
\begin{split}
E_1'(u)&=\rho\lambda_1 E_3(u)\\
E_2'(u)&=0\\
E_3'(u)&=-\rho\lambda_1 E_1(u).
\end{split}\right.
\end{equation*}
Using \eqref{h1}  and \eqref{r1}, 
\begin{equation}\label{sy2}
\left\{
\begin{split}
E_1'(u)&=\frac{e^{2au}}{\rho} E_3(u)=\frac{2AB e^{Bu}}{1+A^2e^{2Bu}} E_3(u)\\
E_2'(u)&=0\\
E_3'(u)&=-\frac{e^{2au}}{\rho} E_1(u)=-\frac{2AB e^{Bu}}{1+A^2e^{2Bu}} E_1(u).
\end{split}\right.
\end{equation}
In particular, $E_2$ is constant. With respect to a fixed orthonormal basis of $\E_1^3$, let 
$E_1(u,0)=(a(u),b(u),0)$, $E_2(u)=(0,0,1)$ and $E_3(u,0)=(b(u),-a(u),0)$. Then the first and third equation of \eqref{sy2}  are equivalent to
\begin{equation*}
\begin{split}
(a',b',0)&=h(b,-a,0)\\
 (b',-a',0)&=-h(a,b,0),
 \end{split}
 \end{equation*}
 where 
 $$h(u)=\frac{2AB e^{Bu}}{1+A^2e^{2Bu}}.$$
Using that $|E_1|^2=|E_3|^2=1$, we have $a'^2+b'^2=h^2$. Thus
$$\frac{a'}{\sqrt{1-a^2}}=h.$$
Solving this equation and since that $b=a'/h$, we have 
$$a(u)= \frac{2 A e^{B u}}{1+A^2 e^{2 B u}},\quad b(u)=\frac{1-A^2 e^{2 B u}}{1+A^2 e^{2 B u}}.$$
With these functions $a$ and $b$, we can integrate \eqref{sy2}, obtaining
\begin{equation}\label{normal-t}
\left\{
\begin{split}
E_1(u)&=\frac{1}{1+A^2 e^{2Bu}}\left(2A e^{Bu},1-A^2 e^{2Bu},0\right)\\
E_2(u)&=(0,0,1)\\
 E_3(u)&=\frac{1}{1+A^2 e^{2Bu}}\left(1-A^2 e^{2Bu},-2A e^{Bu},0\right).
 \end{split}\right.\end{equation}
 Using that $c'(u)=\rho E_1(u)$, the value of $\rho$ in \eqref{r1} and the expression of $E_1(u)$  in \eqref{normal-t}, we obtain the parametric curve $X(u,0)=c(u)$ of the surface.

\begin{lemma}\label{l-t} The space curve $c(u)=X(u,0)$   is given by
\begin{eqnarray*}
c(u)&=&\left(\frac{e^{2au}}{2aB},-\frac{Ae^{(2a+B)u}}{2B(2a+B)}+\frac{e^{(2a-B)u}}{2AB(2a-B)},0\right), \quad B\neq\pm2a\\
c(u)&=& \left(\frac{e^{2au}}{4a^2},-\frac{Ae^{4au}}{16a^2} +\frac{u}{4Aa},0\right), \quad B= 2a\\
c(u)&=&\left(-\frac{e^{2au}}{4a^2},-\frac{e^{4au}}{16 a^2A}+\frac{Au}{4a},0\right), \quad B=- 2a.
\end{eqnarray*}
\end{lemma}

The curve $c(u)$ is  the base curve in the Bj\"{o}rling problem to find a parametrization of the surface $X(u,v)$. Now the solution of the  Bj\"{o}rling problem with initial analytic  data $\{c(u),E_3(u)\}$  is given in terms of paracomplex analysis \cite{cha},
\begin{equation}\label{bs-t}
X(u,v)=\mbox{Re}\left(c(z)+\tau \int^z E_3(w)\times c'(w)\, dw\right),\end{equation}
where $z=u+\tau v$, $\tau^2=1$. Again, $c(u)$ and $E_3(u)$ have paraholomorphic extensions $c(z)$ and $E_3(z)$, respectively, thanks to \eqref{normal-t} and Lemma \ref{l-t}. Replacing in this formula the value of $c(u)$ obtained in Lemma  \ref{l-t} together the expression of $E_3(u)$ in \eqref{normal-t}, we conclude:

\begin{theorem} \label{t38}
Suppose that $\delta=-\epsilon=1$ in \eqref{in1}. Then any timelike intrinsic rotational ZMC surface of $\E_1^3$    and twist $\alpha(v)=av$ is parametrized by
\begin{enumerate}
\item Case $B\not=\pm 2a$,
$$X(u,v)=\left(\begin{array}{c}
\frac{e^{2 a u} \cosh (2 a v)}{2 a B}\\
\frac{e^{(2 a -B) u} \left((2 a+B) \cosh (  (2 a-B)v)-A^2 (2 a-B) e^{2 B u} \cosh (  (2 a+B)v)\right)}{2AB(4 a^2- B^2)}\\
-\frac{e^{(2 a-B) u} \left(A^2 (2 a-B) e^{2 B u} \sinh (  (2 a+B)v)+(2 a+B) \sinh (  (2 a-B)v)\right)}{2AB(4 a^2- B^2)}
\end{array}\right).$$
The first fundamental form is 
$$\textrm{I}=\frac{e^{(4a-2B)u}(A^2e^{2Bu}+1)^2}{4A^2B^2}(du^2-dv^2),$$
and the domain of $X$ is $\r^2$.
\item Case $B=2a$,
 $$X(u,v)=\left(
 \frac{e^{2 a u} \cosh (2 a v)}{4 a^2}, \frac{4 a u-A^2 e^{4 a u} \cosh (4 a v)}{16 a^2 A},
 -\frac{A^2 e^{4 a u} \sinh (4 a v)+4 a v}{16 a^2 A}
 \right).$$
The first fundamental form is 
$$\textrm{I}=\frac{ (A^2e^{4au}+1)^2}{16a^2A^2}(du^2-dv^2),$$
and the domain of $X$ is $\r^2$.

\item Case $B=-2a$,

 $$X(u,v)=\left(
 -\frac{e^{2 a u} \cosh (2 a v)}{4 a^2}, \frac{4 a A^2 u-e^{4 a u} \cosh (4 a v)}{16 a^2 A}, \frac{4 a A^2 v+e^{4 a u} \sinh (4 a v)}{16 a^2 A}
 \right).$$
 The first fundamental form is 
$$\textrm{I}=\frac{ (A^2+e^{4au})^2}{16a^2A^2}(du^2-dv^2),$$
and the domain of $X$ is $\r^2$.
 \end{enumerate}
 In all cases, the domain of $X$ is $\r^2$.
\end{theorem}

The timelike Enneper surface \eqref{ent} corresponds with the case $A=B=a=1$. Again, the cases $B=\pm 2a$ are surfaces where the $1$-forms of the Weierstrass representation have residues.

\subsubsection{Case   $E_1$ is timelike and $E_2$ is spacelike}

The matrix $P$  in \eqref{cc}  is
\begin{equation*} 
P=\left(\begin{array}{ccc}
0&0&\frac{e}{\rho}\\
0&0&-\frac{f}{\rho}\\
\frac{e}{\rho}&\frac{f}{\rho}&0\end{array}\right).
\end{equation*}
At $v=0$, we have $f=0$ and $e(u)=-g(u)=\rho^2\lambda_1(u)$. Then 
\begin{equation*}
\left\{
\begin{split}
E_1'(u)&=\rho\lambda_1 E_3(u)=\frac{2AB}{1-A^2 e^{2Bu}} E_3(u)\\
E_2'(u)&=0\\
E_3'(u)&=\rho\lambda_1 E_1(u)=\frac{2AB}{1-A^2 e^{2Bu}} E_1(u).
\end{split}\right.
\end{equation*}
The integration leads to   
\begin{equation*} 
\left\{
\begin{split}
E_1(u)&=\frac{1}{1-A^2 e^{2Bu}}\left(2A e^{Bu},0,-1-A^2 e^{2Bu}\right)\\
E_2(u)&=(0,1,0)\\
 E_3(u)&=\frac{1}{1-A^2 e^{2Bu}}\left(1+A^2 e^{2Bu},0,-2A e^{Bu}\right)
 \end{split}\right.\end{equation*}
 and 
\begin{eqnarray*}
c(u)&=&\left(\frac{e^{2au}}{2aB},0,-\frac{Ae^{(2a+B)u}}{2B(2a+B)}-\frac{e^{(2a-B)u}}{2AB(2a-B)}\right), \quad B\neq\pm2a\\
c(u)&=& \left(\frac{e^{2au}}{4a^2},0,-\frac{Ae^{4au}}{16a^2} -\frac{u}{4Aa}\right), \quad B= 2a\\
c(u)&=&\left(-\frac{e^{2au}}{4a^2},0,\frac{e^{4au}}{16 a^2A}+\frac{Au}{4a}\right), \quad B=- 2a.
\end{eqnarray*}

\begin{theorem} \label{t39}
Suppose that $\delta=-\epsilon=-1$ in \eqref{in1}. Then any timelike intrinsic rotational ZMC surface of $\E_1^3$  and twist $\alpha(v)=av$ is parametrized by
\begin{enumerate}
\item Case $B\not=\pm 2a$,
$$X(u,v)=\left(\begin{array}{c}
\frac{e^{2 a u} \cosh (2 a v)}{2 a B}\\
-\frac{e^{(2 a-B) u} \left(A^2 (2 a-B) e^{2 B u} \sinh (  (2 a+B)v)-(2 a+B) \sinh (  (2 a-B)v)\right)}{2AB(4 a^2- B^2)}\\
-\frac{e^{(2 a -B) u} \left((2 a+B) \cosh (  (2 a-B)v)+A^2 (2 a-B) e^{2 B u} \cosh (  (2 a+B)v)\right)}{2AB(4 a^2- B^2)}
\end{array}\right).$$
The first fundamental form is 
$$\textrm{I}=-\frac{e^{(4a-2B)u}(A^2e^{2Bu}-1)^2}{4A^2B^2}(du^2-dv^2),$$
and the domain of $X$ is $u\not=\frac{-\log(A)}{B}$, $v\in\r$.

\item Case $B=2a$,
 $$X(u,v)=\left(
 \frac{e^{2 a u} \cosh (2 a v)}{4 a^2}, \frac{4 a v-A^2 e^{4 a u} \sinh (4 a v)}{16 a^2 A},
 -\frac{A^2 e^{4 a u} \cosh (4 a v)+4 a u}{16 a^2 A}
 \right).$$
The first fundamental form is 
$$\textrm{I}=-\frac{ (A^2e^{4au}-1)^2}{16a^2A^2}(du^2-dv^2),$$
and the domain of $X$ is $u\not=\frac{-\log(A)}{2a}$, $v\in\r$.

\item Case $B=-2a$,

 $$X(u,v)=\left(
 -\frac{e^{2 a u} \cosh (2 a v)}{4 a^2}, \frac{4 a A^2 v-e^{4 a u} \sinh (4 a v)}{16 a^2 A}, \frac{4 a A^2 u+e^{4 a u} \cosh (4 a v)}{16 a^2 A}
 \right).$$
The first fundamental form is 
$$\textrm{I}=-\frac{ (A^2-e^{4au})^2}{16a^2A^2}(du^2-dv^2),$$
and the domain of $X$ is $u\not=\frac{\log(A)}{2a}$, $v\in\r$.

 \end{enumerate}
 \end{theorem}
At the points where $X$ is not defined,   the mapping $X$ is not an immersion.
\begin{remark} If $a=A=B=1$ in Theorem \ref{t39}, we find the surface 
\begin{equation}\label{ent2}
X(u,v)=\left(\begin{array}{c}
 \frac{1}{2} e^{2 u} \cosh (2 v)\\
 -\frac{1}{6} e^u \left(e^{2 u} \sinh (3 v)-3 \sinh (v)\right)\\
 -\frac{1}{6} e^u \left(e^{2 u} \cosh (3 v)+3 \cosh (v)\right)\end{array}
 \right).
\end{equation}
This surface is analogous to the timelike Enneper  surface \eqref{ent}. We call the {\it second timelike Enneper surface}. The first fundamental form is 
$$\mathrm{I}=\frac{1}{16} e^{4 u} \left(e^{2 u}-1\right)^4(-du^2+dv^2).$$
Here we exclude the value $u=0$ because in the set  $\{X(0,v)\colon v\in\r\}$, the mapping $X$ is not an immersion.    For example, assuming that the domain of $X$ is $\r^+\times\r$,    the change $e^u\cosh v\to u$ and  $e^u \sinh v\to v$  provides a parametrization of the surface in terms of polynomials on $u$ and $v$, namely, 
$$X(u,v)=\frac12\left(  u^2+v^2 ,-\frac{1}{3} v  (3 u^2+v^2-3 ),-\frac{1}{3} u (u^2+3 v^2+3)\right).$$
 We can obtain the Weierstrass data  after the change of variable $e^z\to z$. Then the Weierstrass data is $g(z)=\tau z$ and $\omega=\tau dz$. The surface $\Sigma$ is $\c$ and the only end is $z=\infty$ with a pole of order $4$, in particular,     the end $z=\infty$ is of Enneper-type.
\end{remark}

\subsection{Untwisted zero mean curvature surfaces}\label{sec5}

In this subsection, we study intrinsic rotational ZMC  surfaces with constant twist $\alpha(v)=c$, $c\in\r$. It will be proved that the surface is associated with a surface of revolution.  

\begin{theorem}\label{t51}
 If $\Sigma$ is an intrinsic rotational ZMC  surface of $\E_1^3$ with twist $\alpha(v)=c$, $c\in\r$, then $\Sigma$ is an associate surface of a ZMC  surface of revolution.
\end{theorem}

\begin{proof}
By Propositions \ref{pr-s} and \ref{pr-t}, we can assume that $c=0$. The proof is achieved if we show that $\Sigma$ is a surface of revolution. We now repeat the process of integration described in Section \ref{sec3}. 
\begin{enumerate}
\item Spacelike case. The function $\rho$ is given in \eqref{r1} for $a=0$.   For the base curve $c(u)=X(u,0)$, we have 
$$c'(u)=\rho(u)E_1(u,0) =\frac{1}{2B}\left(\frac{1}{A}e^{-Bu}+Ae^{Bu},0,2\right).$$
Thus 
$$c(u)=\frac{1}{2B}\left(-\frac{1}{AB}e^{-Bu}+\frac{A}{B}e^{Bu},0,2u\right).$$ 
Using   the Bj\"{o}rling formula \eqref{bs}, the parametrization of the surface is
\begin{eqnarray}\label{untwisted}
 X(u,v)=\left(\begin{array}{c}
-\frac{1}{2AB^2}e^{-Bu}\cos(Bv)+\frac{A}{2B^2}e^{Bu}\cos(Bv)\nonumber\\
-\frac{1}{2AB^2}e^{-Bu}\sin(Bv)+\frac{A}{2B^2}e^{Bu}\sin(Bv)\\
\dfrac{u}{B}
\end{array}\right),\qquad u\not=0.
\end{eqnarray}
If we consider the one-parameter family $\{\Psi_\theta^e:\theta\in\r\}$ of rotations  of $\E_1^3$ about the $x_3$-axis, where
\begin{eqnarray*}
\Psi_\theta^e=\left(\begin{array}{ccc}
\cos\theta&-\sin\theta&0\\
\sin\theta&\cos\theta&0\\
0&0&1
\end{array}\right),
\end{eqnarray*}
it is immediate that    $\Psi_\theta^e(X(u,v))=X(u,v+\theta)$. This proves that $X(u,v)$ is a surface of revolution about the $x_3$-axis. In particular, $X$ is the elliptic catenoid \cite{ko}. Notice that the (horizontal) plane, which is also a spacelike ZMC  surface of revolution about the $x_3$-axis, was discarded because all surfaces considered in this paper have not open sets of umbilic points.  

\item Timelike case. The argument is similar. Now 
$$c'(u)=\rho E_1(u)=\frac{1}{2AB}\left(2A,e^{-Bu}-A^{2}e^{Bu},0\right)$$
and the curve $c(u)$ is
$$c(u)=\frac{1}{2AB}\left(2Au,-\frac{1}{B}e^{-Bu}-\frac{A^{2}}{B}e^{Bu},0\right).$$
The  timelike ZMC  surface obtained by \eqref{bs-t} is
$$X(u,v)=\left(\begin{array}{c}
\dfrac{u}{B}\\
-\frac{1}{2AB^2}\cosh(Bv)e^{-Bu}-\frac{A}{2B^2}\cosh(Bv)e^{Bu}\\
-\frac{1}{2AB^{2}}\sinh(Bv)e^{-Bu}-\frac{A}{2B^{2}}\sinh(Bv)e^{Bu}
\end{array}\right),\qquad (u,v)\in\r^2.$$
If  $\{\Psi_\theta^h:\theta\in\r\}$ is the one-parameter family of hyperbolic rotations  about the $x_1$-axis of $\E_1^3$,   where
\begin{eqnarray*}
\Psi_\theta^h=\left(\begin{array}{ccc}
1&0&0\\
0&\cosh\theta&\sinh\theta\\
0&\sinh\theta&\cosh\theta
\end{array}\right),
\end{eqnarray*}
then    $\Psi_\theta^h(X(u,v))=X(u,v+\theta)$. Therefore   $X(u,v)$ is a surface of revolution about the $x_1$-axis.  This surface is the timelike catenoid (sub-case 1 in Proposition \ref{pr1}): see  \cite[Prop. 1]{lop2}.
\end{enumerate}
\end{proof}

\section{Examples of intrinsic rotational  surfaces with zero mean curvature}\label{sec4}

In this section we show explicit examples of intrinsic rotational ZMC  surfaces. We point out that we are not able to give examples when the mean curvature is a non-zero constant as in \cite{brs} because we do not use loop group methods. 

\subsection{Spacelike surfaces}

Let $\Sigma$ be a  spacelike ZMC   surface described in  Theorem \ref{t33}. We will  give the Weierstrass representation of these surfaces via the holomorphic function $\phi(z)=\frac{\partial X}{\partial z}(z)$ where $z=u+iv$ are conformal coordinates. Recall that  $X(z)=2 \mbox{Re}\int_{z_0}^z\phi(w)\, dw$, and that this integral is made along any path  from $z_0$ to $z$ \cite{ko}. In the case of Theorem \ref{t33}, the constant of integration is determined by the identity $X(u,0)=c(u)$. 

The Weierstrass representation of $\Sigma$ is defined by the pair $(g,\omega)$, where 
$$g(z)=\frac{\phi_3}{\phi_1-i\phi_2},\quad \omega=(\phi_1-i\phi_2) dz,$$
and $\phi=(\phi_1,\phi_2,\phi_3)$. In terms of  $(g,\omega)$, the surface $\Sigma$ is parametrized by 
\begin{equation}\label{wr}
X(u,v)=\mbox{Re}\int_{z_0}^z\left(\frac12(1+g^2)\omega,\frac{i}{2}(1-g^2)\omega,g\omega\right).
\end{equation}
If $\phi_k$ do not have real periods, the integrals $\mbox{Re}\int^z\phi_k\, dz$ are well defined independently of the path from the initial point $z_0$ to $z$. From the Bj\"{o}rling formula \eqref{bs} we have
$$\phi(z)= c'(z)+i (E_3(z)\times c'(z)).$$
A computation of $g$ and $\omega$ gives
$$g(z)=\frac{e^{-Bz}}{A},\quad \omega=\frac{A e^{(2a+B)z}}{2B}dz$$
in all cases of $A$, $B$ and $a$. As far as the authors know, these surfaces do not appear in the literature and only one particular case was discussed in \cite{gu}. Some references of examples of spacelike ZMC  surfaces in $\E_1^3$ are \cite{af,fu,kim,ky,ko,lk,wo}, without to be a complete list. 

We give some examples of surfaces of this family, focusing on the cases when $B$ and $2a$  are integers. Let $n=B$ and $m=2a$, $n,m\in\mathbb{Z}$.  So the change $z\to-\log(z)$ and next, $z\to\lambda z$ shows that it is possible to assume $A=1$.  Then $\omega=-1/(2n) z^{-n-m-1}dz$.  After a dilation of $\Sigma$ in $\E_1^3$, the Weierstrass representation is  
\begin{equation}\label{gw}
g(z)=z^n,\quad \omega= \frac{1}{z^{n+m+1}}dz.
\end{equation}
 Let $\mathcal{E}_s(n,m)$ denote the corresponding spacelike ZMC  surface with these data $(g,\omega)$. In particular, $\mathcal{E}_s(1,2)$ is  the spacelike Enneper surface \eqref{ens}  and $\mathcal{E}_s(1,0)$ is the elliptic catenoid.  On the other hand, the surface $\mathcal{E}_s(n,-n-1)$  appeared in  \cite[Lemma 1]{gu}.
 
Using \eqref{wr}, a parametrization of $\mathcal{E}_s(n,m)$ is  
$$X(u,v)= \mbox{Re}\int_{z_0}^z\left(\frac12\left( z^{-n-m-1}+z^{n-m-1}\right), 
\frac{i}{2}\left( z^{-n-m-1}-z^{n-m-1}\right), z^{-m-1}\right)\, dz.$$
Notice that   the real periods appear in the $1$-form $\phi_2\, dz$, namely,
$$\phi_2 dz=\frac{i}{2}\left(\frac{1}{z^{n+m+1}}+\frac{1}{z^{-n+m+1}}\right)dz.$$
 This occurs only when $n+m=0$ or $n-m=0$.  In terms of the constants $B$ and $a$ in Theorem \ref{t33}, both cases correspond with $B=\pm 2a$.  We now present some explicit examples of surfaces.  We will also study the geometry of their ends which, as in Euclidean case,  is determined by the order of the poles of $\phi$ at these points \cite{lm}. If the maximum of the orders of $\phi$ at an end $p$ is $k$, then it is of Enneper type if $k=4$. If $k=2$, then it is of planar or catenoid type depending if the residue is $0$ or not, respectively.

\begin{enumerate}

\item The surface $\mathcal{E}_s(2,-1)$ (Figure \ref{fig3}). This case corresponds with $n+m=0$. The Weierstrass data $g(z)=z^2$ and $\omega=dz/z^2$ are defined on $\Sigma=\c-\{0\}$. We analyze the two ends $z=0$ and $z=\infty$:  
$$\max\{\mbox{ord}(\phi_1,0),\mbox{ord}(\phi_2,0),\mbox{ord}(\phi_3,0)\}=\max\{2,2,0\}=2,$$
$$\max\{\mbox{ord}(\phi_1,\infty),\mbox{ord}(\phi_2,\infty),\mbox{ord}(\phi_3,\infty)\}=\max\{4,4,2\}=4.$$
Thus $z=0$ is an end of catenoid or planar type. Since $g(z)=z^2$ is not bijective around $z=0$,   then $z=0$ is end of planar type. On the other hand, the end $z=\infty$ is of order $4$, so it is of Enneper type.

\item The surface $\mathcal{E}_s(2,2)$. The Weierstrass data are $g(z)=z$ and $\omega=dz/z^5$. Since $n-m=0$, the $1$-form $\phi_2\ dz$ has real periods. In particular, the surface is periodic along the $x_1$-axis: see Figure \ref{fig4}.
\item The surface $\mathcal{E}_s(1,-1)$. This case corresponds with $n+m=0$. The Weierstrass data $g(z)=z$ and $\omega=dz/z$ are defined on $\Sigma=\c-\{0\}$. We have $\max\{\mbox{ord}(\phi_i,0)\colon 1\leq i\leq 3\}=1$ and  $\max\{\mbox{ord}(\phi_i,\infty)\}=3$. Now $\phi_1\ dz$ has real periods.  see Figure \ref{fig4}.
\end{enumerate}

\begin{figure}[hbtp]\centering
 \includegraphics[width=.5\textwidth]{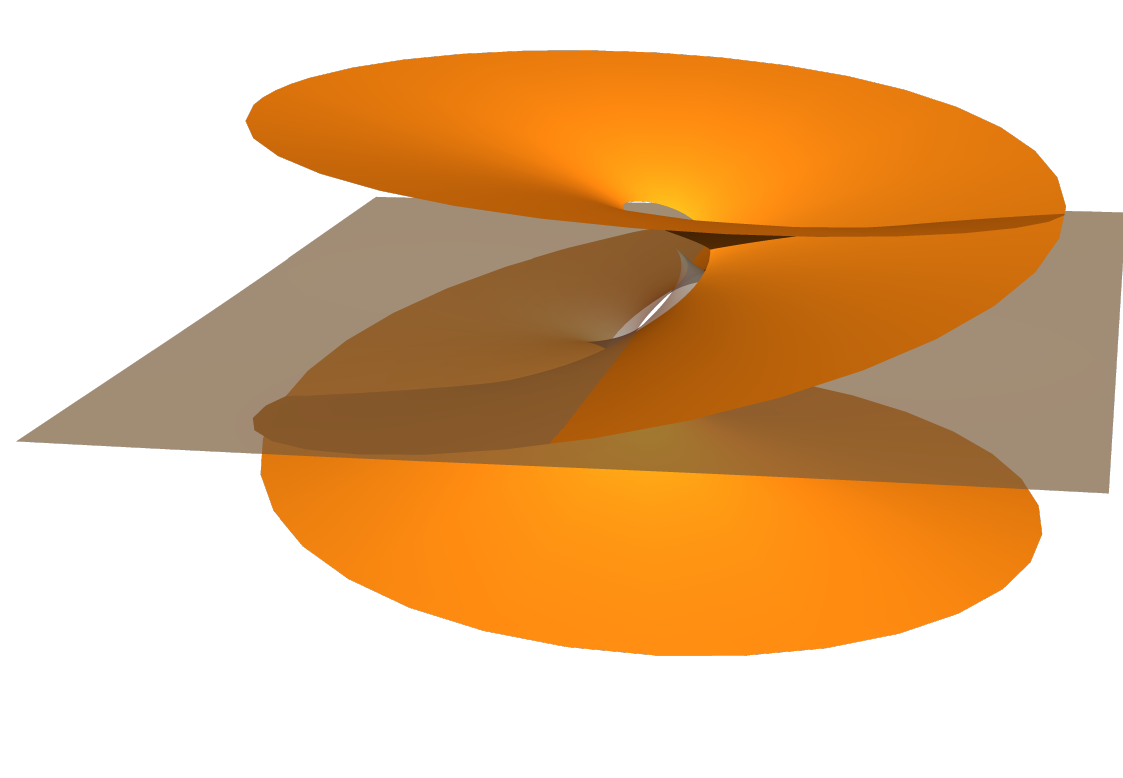} 
\caption{The spacelike ZMC  surface $\mathcal{E}_s(2,-1)$.  The surface presents an Enneper type end at $z=\infty$ and a planar end at  $z=0$}\label{fig3}
\end{figure}
\begin{figure}[hbtp]
\centering
 \includegraphics[width=.5\textwidth]{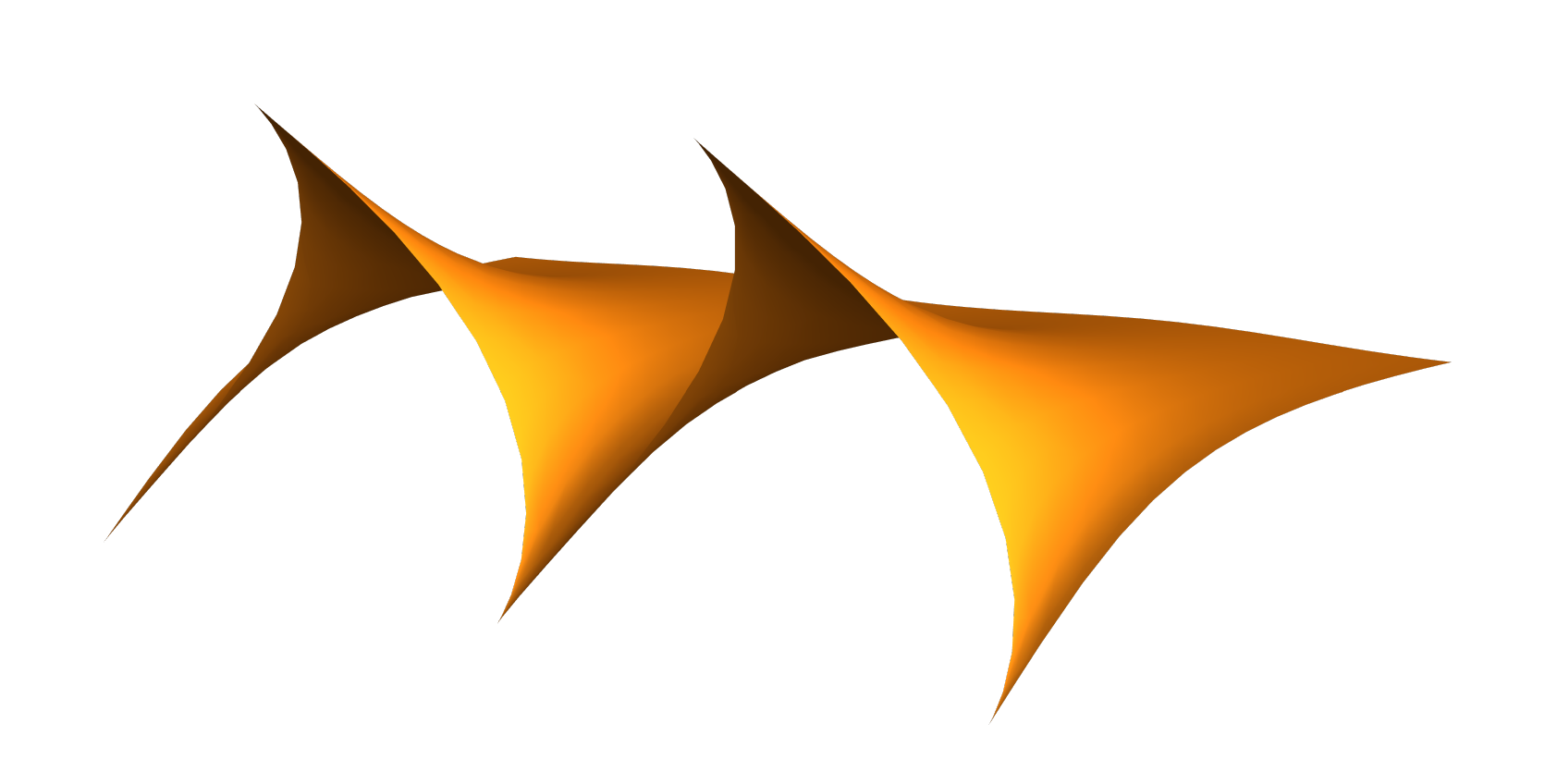}\qquad \includegraphics[width=.3\textwidth]{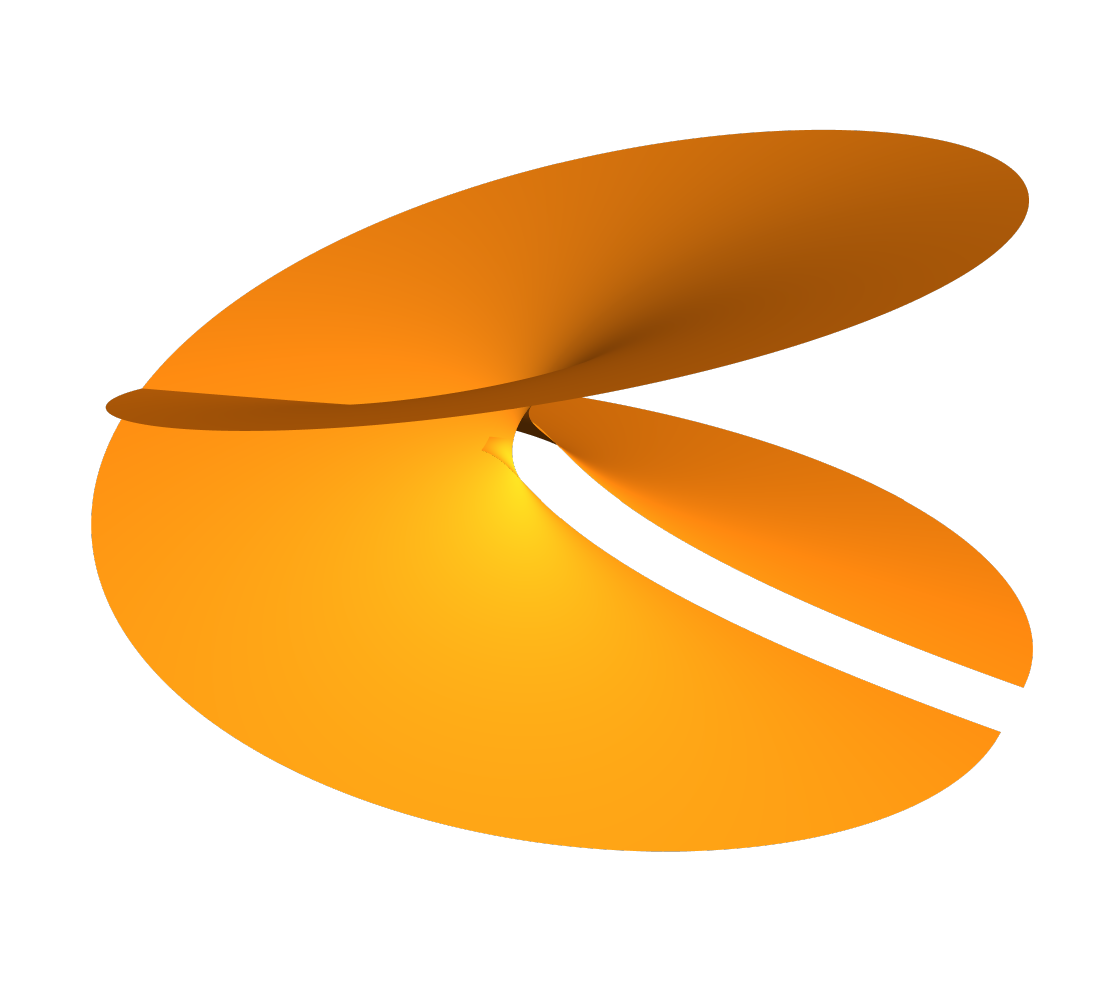}
\caption{The spacelike ZMC  surfaces $\mathcal{E}_s(2,2)$ (left) and $\mathcal{E}_s(2,-2)$ (right).}\label{fig4}
\end{figure}

\subsection{Timelike surfaces}\label{sec42}

As we have seen in Section \ref{sec3}, there are two types of timelike ZMC  surfaces. In order to give some examples of such surfaces, we only describe those of  Theorem \ref{t38}. The Weierstrass representation of $\Sigma$ is defined thanks to the paraholomorphic function $\phi(z)=\frac{\partial X}{\partial z}(z)$, where now $z=u+\tau v$ is a paraconformal coordinate on $\Sigma$. If $\phi=(\phi_1,\phi_2,\phi_3)$, then 
the Weierstrass representation is defined by  $(g,\omega)$, where 
$$g(z)=\frac{\phi_1}{\phi_2-\tau\phi_3},\quad \omega=(\phi_2-\tau\phi_3) dz,$$
 and the parametrization of the surface is 
 $$X(u,v)=\mbox{Re}\int_{z_0}^z\left(g\omega,\frac12(1-g^2)\omega,\frac{\tau}{2}(1+g^2)\omega \right).$$
 See \cite{kon}. If the timelike surface is given by the Bj\"{o}rling formula \eqref{bs-t}, then the paraholomorphic function $\phi(z)$ is
$$\phi(z)=c'(z)+\tau(E_3(z)\times c'(z)).$$
 The Weierstrass data are 
 $g(z)=Ae^{Bz}$ and $w=\frac{e^{(2a-B)z}}{AB}dz$ in all cases. We will study the particular situation that $B$ and $2a$ are integers, $B=n$, $2a=m$, $n,m\in\mathbb{Z}$. As in the spacelike case, after the change $z\to \log(z)$, a dilation in the parameter $z$ and a dilation in the ambient space $\E_1^3$, we have 
 $$g(z)= z^n,\quad \omega=\frac{1}{z^{n-m+1}}\ dz.$$
  Let $\mathcal{E}_t(n,m)$ denote the   timelike ZMC  surface with the above Weierstrass representation $(g,\omega)$. 
With these values of the Weierstrass representation, the surface  is given by
$$X(u,v)=\mbox{Re}\int_{z_0}^z\left( z^{m-1},\frac{1}{2} (z^{m-n-1}-  z^{m+n-1}),\frac{\tau}{2} (z^{m-n-1}+  z^{m+n-1})
 \right)dz.$$
The real periods appear for   $\phi_3 dz$  which occurs when $n-m=0$ or $n+m=0$. Notice that   $\mathcal{E}_t(1,2)$ is the timelike Enneper surface \eqref{ent}  and $\mathcal{E}_t(1,0)$ is the timelike catenoid.
\begin{enumerate}
 
\item The timelike surface $\mathcal{E}_t(2,1)$ has Weierstrass data    $g(z)=z^2$ and $\omega=dz/z^2$. Now  $\Sigma=\c-\{0\}$. The points $z=0$ and $z=\infty$ are the two ends of the surface. The orders of the poles are
$$\max\{\mbox{ord}(\phi_1,0),\mbox{ord}(\phi_2,0),\mbox{ord}(\phi_3,0)\}=\max\{0,2,2\}=2,$$
$$\max\{\mbox{ord}(\phi_1,\infty),\mbox{ord}(\phi_2,\infty),\mbox{ord}(\phi_3,\infty)\}=\max\{2,4,4\}=4.$$
At $z=0$, the Gauss map $g(z)=z^2$ is not one-to-one. Since the order is $2$, then $z=0$ is an end of planar type. On the other hand,    $z=\infty$ is an end of Enneper type because its order is $4$.  A parametrization of the surface is 
$$X(u,v)=\frac16 e^u\left(\begin{array}{c}
3 e^{u} \cosh (2 v)\\
 - e^{2u}   \cosh (3 v)+3 \cosh (v) \\
 e^{2u} \sinh (3 v)+3 \sinh (v) 
 \end{array}\right),\qquad u,v\in\r.$$
\item The timelike surface $\mathcal{E}_t(1,-1)$ has Weierstrass data    $g(z)=z$ and $\omega=dz/z^3$. In this case, the $1$-form 
$$\phi_3 dz=\frac{\tau}{2}\left(z+\frac{1}{z}\right)dz$$
 has real periods. See Figure \ref{fig5}.  A parametrization of the surface is 
$$X(u,v)= \left(\begin{array}{c}
e^u \cosh (v)\\
\frac{1}{4} \left(2 u-e^{2 u} \cosh (2 v)\right)\\
\frac{1}{2} \left(e^{2 u} \sinh (v) \cosh (v)+v\right)
 \end{array}\right),\qquad u,v\in\r.$$
\item The timelike surface $\mathcal{E}_t(1,1)$ has Weierstrass data    $g(z)=z$ and $\omega=dz/z$. We have $\max\{\mbox{ord}(\phi_i,0)\colon 1\leq i\leq 3\}=1$ and $\max\{\mbox{ord}(\phi_i,\infty)\}=3$. Now $\phi_1\ dz$ has real periods.  see Figure \ref{fig5}.
 In this case, the $1$-form 
$$\phi_3 dz=\frac{\tau}{2}\left(z+\frac{1}{z}\right)dz$$
 has real periods. See Figure \ref{fig5}.  
\end{enumerate}
\begin{figure}[hbtp]
\centering
 \includegraphics[width=.4\textwidth]{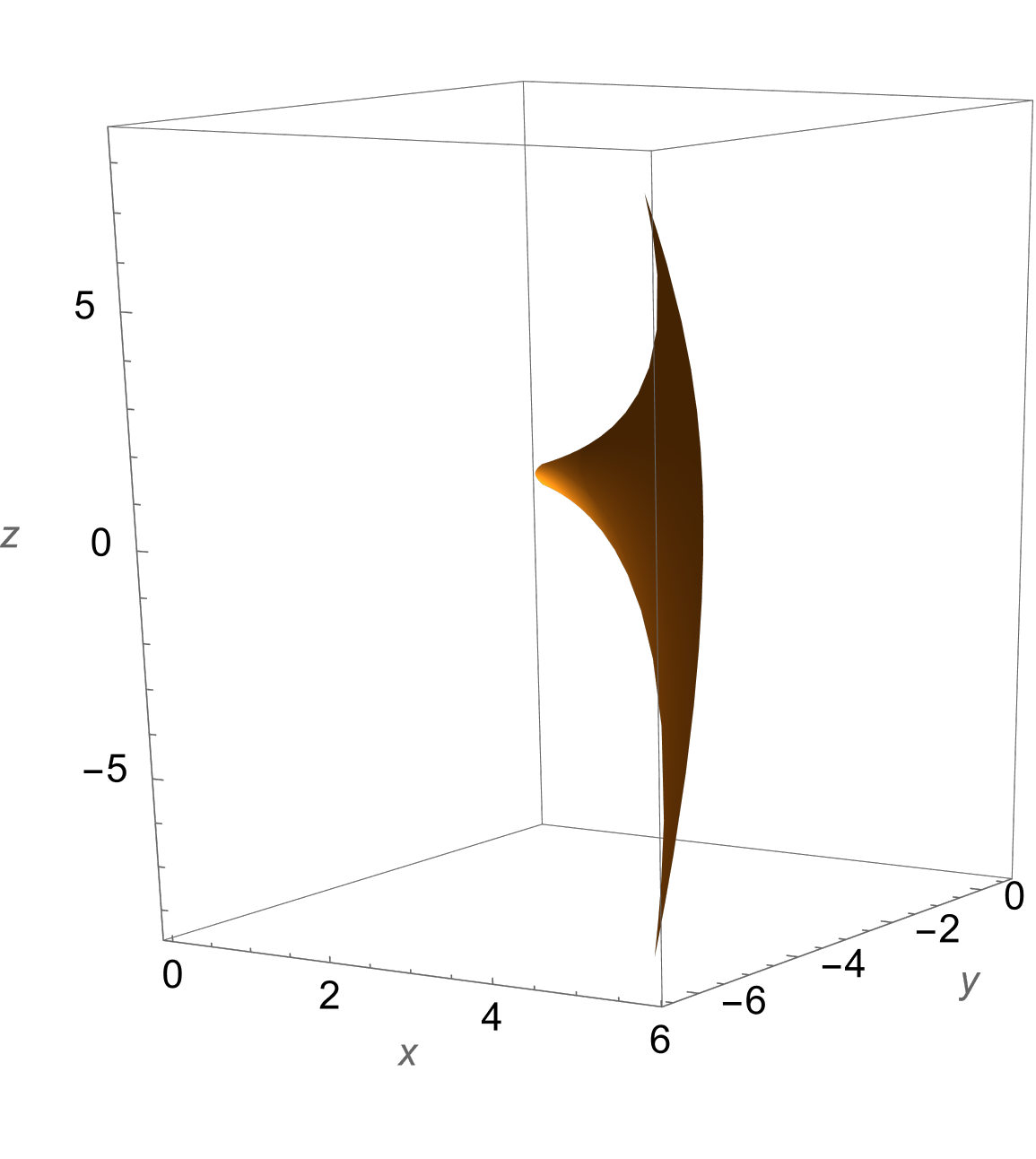}\hspace*{1cm}
\includegraphics[width=.25\textwidth]{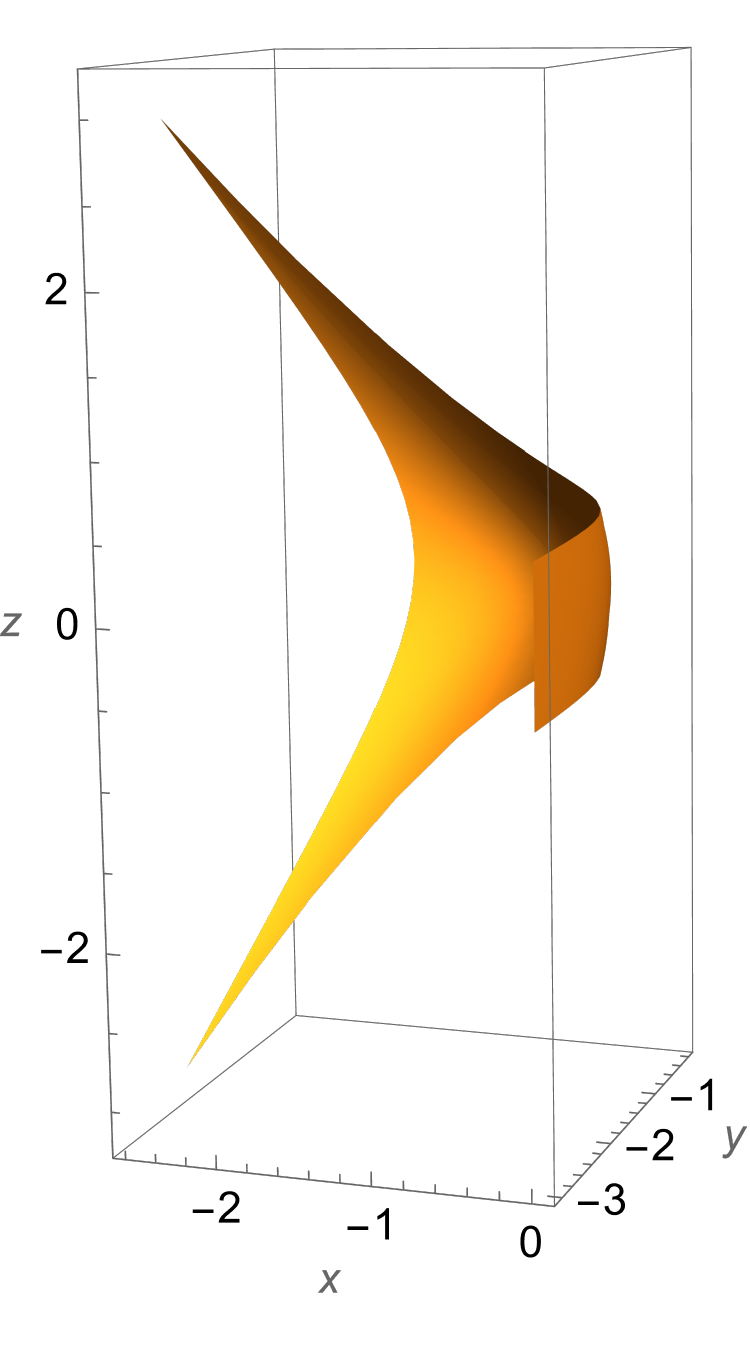}\hspace*{1cm}
\includegraphics[width=.2\textwidth]{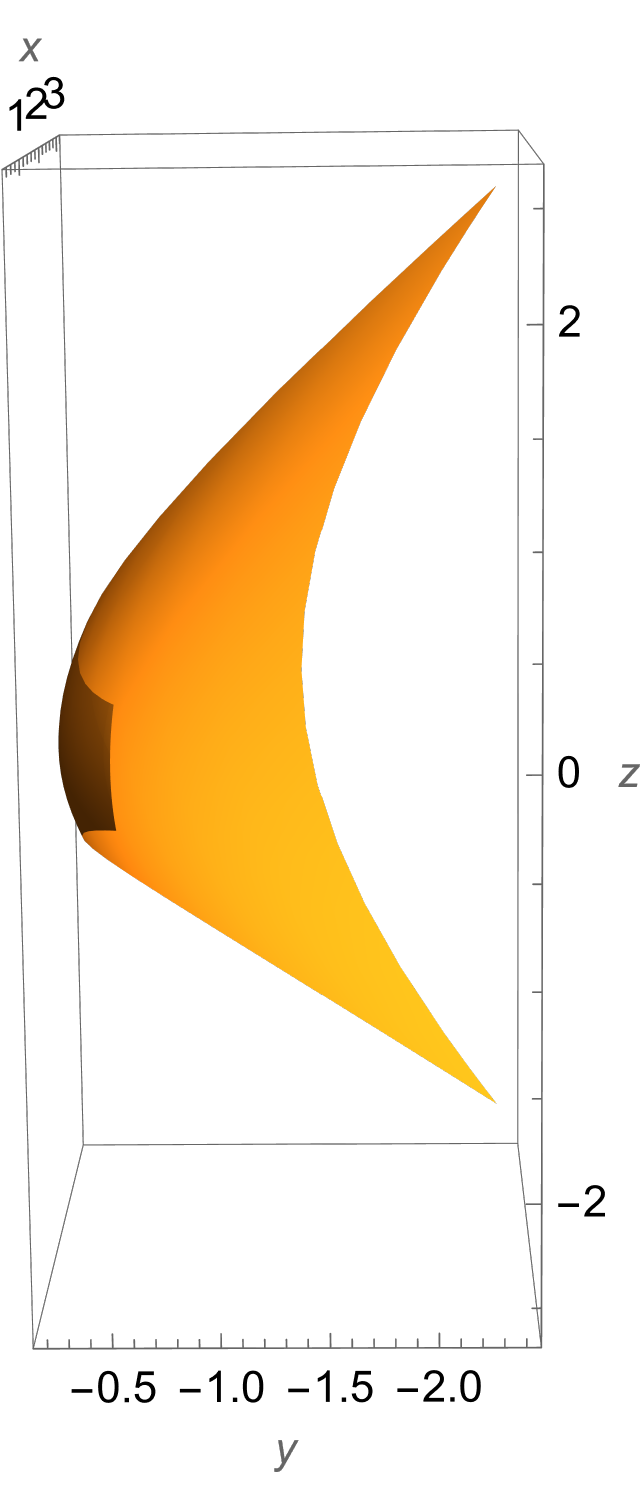}
\caption{The  timelike ZMC  surfaces $\mathcal{E}_t(1,2)$ (left), $\mathcal{E}_t(1,-1) $ (middle) and $\mathcal{E}_t(1,1)$ (right). }\label{fig5}
\end{figure}

\section*{Acknowledgement}
 The authors would like to thank the anonymous referees for insightful suggestions, which helped to improve the exposition of the paper. 
 
  Rafael  L\'opez   is a member of the IMAG and of the Research Group ``Problemas variacionales en geometr\'{\i}a'',  Junta de Andaluc\'{\i}a (FQM 325). This research has been partially supported by MINECO/MICINN/FEDER grant no. PID2023-150727NB-I00,  and by the ``Mar\'{\i}a de Maeztu'' Excellence Unit IMAG, reference CEX2020-001105- M, funded by MCINN/AEI/10.13039/ 501100011033/ CEX2020-001105-M.


\begin{thebibliography}{99}

\bibitem{ak} S. Akamine, J. Cho and Y. Ogata,   {\it Analysis of timelike Thomsen surfaces}, J. Geom. Anal. {\bf 30}, 731--761, 2020.

\bibitem{af} S. Akamine and H. Fujino, {\it Reflection principle for lightlike line segments on maximal surfaces}, Ann. Global Anal. Geom. {\bf 59}, 93--108,  2021.
 

 
\bibitem{acm}  L. J. Al\'{\i}as, R. M. B.    Chaves and P.   Mira, {\it   Bj\"{o}rling problem for maximal surfaces in Lorentz-Minkowski space}, Math. Proc. Camb. Phil. Soc. {\bf 134}, 289--316, 2003.

\bibitem{bi} L. Bianchi, {\it Lezioni di Geometria Differenziale}, Volume II. Enrico Spoerri, Pisa, 1903.


\bibitem{bo}  A. I. Bobenko,    {\it  Constant mean curvature surfaces and integrable equations}, Uspekhi Mat. Nauk. Russian Math. Surv. {\bf 46},   3--42, 1991.

 \bibitem{brs} D. Brander, R.   Rossman and N.  Schmitt,  {\it   Holomorphic representation of constant mean curvature surfaces in Minkowski space: Consequences of non-compactness in loop group methods},  Adv. Math. {\bf 223},  949--986, 2010.



\bibitem{cha}  R. M. B. Chaves, M. P.   Dussan and M.  Magid, M.  {\it   Bj\"{o}rling problem for timelike surfaces in the Lorentz-Minkowski space}, J. Math.  Anal.  Appl. {\bf 377}, 481--494, 2011.

\bibitem{da} L. C. B.  Da Silva,  
 {\it  Surfaces of revolution with prescribed mean and skew curvatures in Lorentz-Minkowski space}, Tohoku Math. J.  {\bf 73},    317--339, 2021.

\bibitem{di} U. Dierkes, S. Hildebrandt, A. K\"{u}ster and O. Wohlrab,  {\it Boundary Value Problems. Minimal Surfaces I.} Springer, Berlin, 1992.

\bibitem{er} S. Erdem,    {\it  Harmonic maps of Lorentz surfaces, quadratic differentials and paraholomorphicity}, Beitr\"{a}ge Algebra Geom.  {\bf 38}, 19--32, 1997.

\bibitem{es} F. J. M. Estudillo and A. Romero, Generalized maximal surfaces in Lorentz-Minkowski space $L^3$, Math. Proc. Camb. Phil. Soc. {\bf 111},  515--524, 1992.

\bibitem{fw}   D. Freese and M. Weber,    {\it  On surfaces that are intrinsically surfaces of revolution}, J. Geom. {\bf 108},  743--762, 2017.



\bibitem{fu}  S. Fujimori,  Y. W.   Kim, S-E-  Koh, W.  Rossman, H.  Shin, H.  Takahashi, M.  Umehara, K.  Yamada and S-D. Yang,   {\it  Zero mean curvature surfaces in $L^3$ containing a light-like line}, C. R. Math. Acad. Sci. Paris {\bf 350},  975--978, 2012.


\bibitem{fsuy} S. Fujimori, K. Saji, M. Umehara and K. Yamada, {\it Singularities of maximal surfaces}, Math. Z. {\bf 259},  827--848, 2008.

\bibitem{gu}  E. G\"{u}ler,   {\it  The algebraic surfaces of the Enneper family of maximal surfaces in three dimensional Minkowski space}, Axioms {\bf 11},  4, 2022.

\bibitem{it} J. Inoguchi and M. Toda, Timelike minimal surfaces via loop groups, Acta Appl. Math. {\bf 63}, 313--355, 2004.




\bibitem{kim}  Y. W.  Kim, S.  Koh, H.  Shin and S.    Yang,     {\it  Spacelike maximal surfaces, timelike minimal surfaces and Bj\"{o}rling representation formulae}, J. Korean Math. Soc. {\bf 48}, 1083--1100,  2011.

\bibitem{ky}  Kim, Y. W.;  Yang, S.-D.  {\it  A family of maximal surfaces in Lorentz-Minkowski three-space}, Proc. Amer. Math. Soc. {\bf 134},  3379--3390, 2006.

\bibitem{ko}  O. Kobayashi,     {\it  Maximal surfaces in the 3-dimensional Minkowski Space $L^{3}$}, Tokyo J. Math. {\bf 6}, 297--309, 1983.

\bibitem{kon} J. J.  Konderak,   {\it  A Weierstrass representation theorem for Lorentz surfaces}, Complex Var. Theory Appl. {\bf 50},    319--332, 2005.

\bibitem{kta}  Z. Kose, M. Toda and E.  Aulisa,    {\it  Solving Bonnet problems to construct families of surfaces}, Balk. J. Geom. Appl. {\bf 16}, 70--80,  2011.

\bibitem{lop2}    L\'opez, R.  {\it  Timelike surfaces with constant mean curvature in Lorentz three-space}, Tohoku Math. J. {\bf 52}, 515--532, 2000.
 

\bibitem{lop}  R.  L\'opez,    {\it  Differential geometry of curves and surfaces in Lorentz-Minkowski space}, Int.   Electron. J. Geom. {\bf 7},  44--107,  2014.

\bibitem{lo}  R. L\'opez,    {\it  Surfaces in Lorentz-Minkowski space with mean curvature and Gauss curvature both constant}, In: Differential Geometry in Lorentz-Minkowski Space,  Ed. Univ. Granada, Granada, 2017, pp. 71--85.

\bibitem{lk} R. L\'opez and S.  Kaya,   {\it  New examples of maximal surfaces in Lorentz-Minkowski space}, Kyushu J. Math. {\bf 71},  311--327, 2017.

\bibitem{lm}  F. J. L\'opez and F. Mart\'{\i}n, {\it Complete minimal surfaces in $R^3$},  Publ. Mat. {\bf 43}, 341--449, 1999.
\bibitem{lp} R.  L\'opez and P\'ampano,    {\it  Classification of rotational surfaces with constant skew curvature in 3-space forms}, J. Math. Anal. Appl. {\bf 489}, 124195,  (2020.

\bibitem{mi} T. K. Milnor, Entire timelike minimal surfaces in $E^3_1$, Michigan Math. J. {\bf 37}, 163-- 177, 1990.


\bibitem{og}  Y. Ogata,    {\it Spacelike constant mean curvature and maximal surfaces in $3$-dimensional de Sitter space via Iwasawa splitting}, Tsukuba J. Math. {\bf 39}, 259--284, 2016.

\bibitem{og2}  Y. Ogata,    {\it  The DPW method for constant mean curvature surfaces in 3-dimensional Lorentzian spaceforms, with applications to Smyth type surfaces}, Hokkaido Math. J. {\bf 46}, 315--350,  2017.

\bibitem{sm}  B. Smyth,    {\it  A generalization of a theorem of Delaunay on constant mean curvature surfaces},
Statistical Thermodynamics and Differential Geometry of Microstructured Materials (Minneapolis, MN, 1991), IMA Vol. Math. Appl., vol. 51, Springer, New York, pp. 123--130, 1993

\bibitem{te} C-L. Terng,  {\it  Lecture Notes on Curves and Surfaces}, Part I. Univ.   of California, 2005. \url{https://www.math.uci.edu/~cterng/162A_Lecture_Notes.pdf}
 
\bibitem{tp}  M. Toda and A.  Pigazzini,    {\it  A note on the class of surfaces with constant skew curvature}, J. Geom. Symmetry Phys. {\bf 46}, 51--58, 2017.
 
 \bibitem{ti}  M. Timmreck, U.  Pinkall and D.  Ferus,    {\it  Constant mean curvature planes with inner rotational symmetry in Euclidean 3-space}, Math. Z. {\bf 215}, 561--568, 1994.
 
\bibitem{wo}    I. Van de Woestijne,   {\it  Minimal surfaces of the 3-dimensional Minkowski space}, In: Geometry and topology of submanifolds, II (Avignon, 1988),   Teaneck, NJ, USA:  World Sci Publ. 1990, pp. 344--369.


\bibitem{uy} M. Umehara and K. Yamada, Maximal surfaces with singularities in Minkowski space, Hokkaido Math. J.  {\bf 35},    13--40, 2006.

\bibitem{we} T. Weinstein, An Introduction to Lorentz Surfaces, de Gruyter Exposition in Math. 22, Walter de Gruyter, Berlin, 1996.

\end{thebibliography}
\end{document}